\documentclass{amsart}

\usepackage{amsmath,amscd,amssymb,amsthm,amsbsy,amscd,stmaryrd}
\usepackage[dvips]{graphicx,color}
\usepackage{epic,eepic,pstricks}
\usepackage{epsfig}
\usepackage{mathrsfs}
\usepackage{rotating}
\usepackage{mathpazo}
\usepackage[varumlaut]{yfonts}
\usepackage{fancybox}

\input{xy}

\usepackage{pstricks}
\usepackage{color}
\usepackage{graphicx}
\usepackage[all]{xy}
\usepackage {hyperref}
\usepackage{framed, color}
\definecolor{shadecolor}{rgb}{1,0.8,0.3}

\newtheorem{theorem}{Theorem}[section]
\newtheorem{lemma}[theorem]{Lemma}

\newtheorem{corollary}[theorem]{Corollary}

\theoremstyle{definition}

\newtheorem{remark}[theorem]{Remark}

\numberwithin{subsection}{section}
\numberwithin{equation}{section}

\begin{document}
%%%%%%%%%%%%%%%%%%%%%%%%%%%%%%%%%%%%%%%%%%%%%%%%%%%%%%%%%%%%%%%%%%%%%%%%%%%%%%
\title{Local curvature
estimates along the $\kappa$-LYZ flow}
%%%%%%%%%%%%%%%%%%%%%%%%%%%%%%%%%%%%%%%%%%%%%%%%%%%%%%%%%%%%%%%%%%%%%%%%%%%%%%
\author{Yi Li}

\address{School of Mathematics and Shing-Tung Yau Center of Southeast
University, Southeast University, Nanjing 211189, China}

\email{yilicms@gmail.com; yilicms@seu.edu.cn}

\author{Yuan Yuan}

\address{Department of Mathematics, Syracuse University, Syracuse,
NY 13244, USA}

\email{yyuan05@syr.edu}

\thanks{The first author is supported in part by start-up funding
of Southeast University $\#$ 4307012071; the second author is supported in part by NSF grant DMS-1412384 and Simons Foundation grant $\#$ 429722}

%%%%%%%%%%%%%%%%%%%%%%%%%%%%%%%%%%%%%%%%%%%%%%%%%%%%%%%%%%%%%%%%%%%%%%%%%%%%%%
\begin{abstract} In this paper we prove a local curvature estimate for the $\kappa$-LYZ flow over K\"ahler manifolds introduced in \cite{FGP1} and  \cite{LYZ}. In particular, we generalize the long time existence of the flow.
\end{abstract}
%%%%%%%%%%%%%%%%%%%%%%%%%%%%%%%%%%%%%%%%%%%%%%%%%%%%%%%%%%%%%%%%%%%%%%%%%%%%%%
\maketitle

%\tableofcontents

%%%%%%%%%%%%%%%%%%%%%%%%%%%%%%%%%%%%%%%%%%%%%%%%%%%%%%%%%%%
%%%%%%%%%%%%%%%%%%%%%%%%%%%%%%%%%%%%%%%%%%%%%%%%%%%%%%%%%%%
\section{Introduction}\label{section1}
%%%%%%%%%%%%%%%%%%%%%%%%%%%%%%%%%%%%%%%%%%%%%%%%%%%%%%%%%%%
%%%%%%%%%%%%%%%%%%%%%%%%%%%%%%%%%%%%%%%%%%%%%%%%%%%%%%%%%%%

%We continue to study the geometric flow introduced in \cite{LYZ}.
Let $(X,
\omega)$ be a closed K\"ahler manifold of complex dimension $n$.
One of the central problems in K\"ahler geometry is the existence of the constant scalar curvature K\"ahler  (cscK) metrics in the K\"ahler class $[\omega]$.
Although Chen and Cheng [CC1-3] made a breakthrough on the problem using the elliptic approach, the parabolic approach still remains open.
In particular, the long time existence and convergence of the Calabi flow is not known in the general case (cf. references in \cite{LYZ} for studies of the Calabi flow).  Since the Calabi flow is a fully nonlinear fourth order partial differential equation, Yuguang Zhang and the authors in the present paper introduced the following coupled flow  in \cite{LYZ}:
\begin{equation}
\partial_{t}\omega(t)=-{\rm Ric}(\omega(t))
-\omega(t)+\alpha(t), \ \partial_{t}\alpha(t)=\overline{\Box}_{\omega(t)}
\alpha(t), \ (\omega,\alpha)(0)=(\omega,\alpha),\label{1.1}
\end{equation}
where $\alpha$ is a closed Hermitian $(1,1)$-form, ${\rm Ric}(\omega(t))$
is the Ricci form of $\omega(t)$, and $\overline{\Box}_{\omega(t)}$
is the complex Hodge-Laplace operator. The motivation of defining such flow is to reduce the fully nonlinear fourth order equation to a system of better understood equations, namely, the K\"ahler-Ricci flow and the heat flow. Note that the stationary solution to (\ref{1.1}) is the cscK metric coupled with a harmonic $(1,1)$-form.

In \cite{LYZ}, the following long time existence result is obtained by using Shi type estimates.
Let $(\omega(t),\alpha(t))$ be the unique solution on $[0,T_{\max})$ for some maximal
time $T_{\max}>0$.

%The short-time existence of (\ref{1.1}) was proved in \cite{LYZ}. Namely,
%there is a . Moreover, the long time existence result is also obtained.

\begin{theorem}\label{t1.1}{\bf (\cite{LYZ})} If $T_{\max}<+\infty$, then
\begin{equation}
\limsup_{t\to T_{\max}}\max_{X}\left\{|{\rm Rm}(\omega(t))|_{\omega(t)},
|\alpha(t)|_{\omega(t)}\right\}=+\infty.\label{1.2}
\end{equation}
\end{theorem}

Since the flow (\ref{1.1}) preserves the
 the closedness of $\omega(t)$ and $\alpha(t)$, it follows that $[\omega(t)]
=[\omega]$ and $[\alpha(t)]=[\alpha]$ (cf. \cite{LYZ}) under the cohomological condition
\begin{equation}
-2\pi c_{1}(X)+[\alpha]=[\omega].\label{1.3}
\end{equation}
By $\partial\overline{\partial}$-lemma, we write
\begin{equation}
\omega(t)=\omega+\sqrt{-1}\partial\overline{\partial}
\varphi(t), \ \ \ \alpha(t)=\alpha+\sqrt{-1}
\partial\overline{\partial}f(t)\label{1.4}
\end{equation}
for smooth functions $\varphi(t)$ and $f(t)$ on $X$. By (\ref{1.3}), we have
\begin{equation}
\omega=\alpha+\sqrt{-1}\partial\overline{\partial}\ln\Omega\label{1.5}
\end{equation}
for a smooth volume form $\Omega$ on $X$. Therefore,
the flow (\ref{1.1}) is equivalent to the following parabolic complex Monge-Amp\`ere
equation coupled with the heat equation
\begin{equation}
\partial_{t}\varphi(t)=\ln\frac{(\omega+\sqrt{-1}\partial\overline{\partial}
\varphi(t))^{n}}{\Omega}-\varphi(t)+f(t), \ \ \ \partial_{t}f(t)
=\Delta_{\omega(t)}f(t)+{\rm tr}_{\omega(t)}\alpha ,\label{1.6}
\end{equation}
with $(\varphi, f)(0)=(0,0)$, where $\Delta_{\omega(t)}$
stands for the complex Laplacian on $X$. The following long time existence result is obtained by assuming (\ref{1.3}).

\begin{theorem}\label{t1.2}{\bf (\cite{LYZ})} Under the condition (\ref{1.3}), if $\alpha$ is nonnegative, and
the Ricci curvature of $\omega(t)$ and $|\alpha(t)|_{\omega(t)}$ are uniformly
bounded on $[0,T)$ with $T<+\infty$, then the solution $(\omega(t),
\alpha(t))$ of (\ref{1.1})
 extends past time $T$.
\end{theorem}

Inspired by M theory, Teng Fei, Bin Guo and D. H. Phong in \cite{FGP1} introduced the $\kappa$-LYZ flow :
\begin{equation}
\partial_{t}\omega(t)=-{\rm Ric}(\omega(t))
+\lambda\!\ \omega(t)+\alpha(t), \ \partial_{t}\alpha(t)
=\kappa\!\ \overline{\Box}_{\omega(t)}\alpha(t), \ (\omega,
\alpha)(0)=(\omega,\alpha),\label{1.7}
\end{equation}
with constants $\kappa>0$ and $\lambda\in\mathbb{R}$. When $(\kappa,
\lambda)=(1,-1)$, the flow (\ref{1.7}) reduces to (\ref{1.1}). Assume the cohomology condition
\begin{equation}
-2\pi c_{1}(X)+[\alpha]+\lambda[\omega]=0.\label{1.8}
\end{equation}
For $\kappa\neq1$, they proved that higher order estimates for $(\omega(t),
\alpha(t))$ follow from $C^{0}$ estimate for $(\omega, \alpha)$. Moreover, they derived very interesting convergence results. For example, on the Riemann surface with negative Euler characteristic,  the $\kappa$-LYZ flow converges in $C^\infty$ topology to the cscK metric in the case of $\kappa\not=1$. Discovered by Fei-Guo-Phong, the $\kappa$-LYZ flow, as the simplest model in the coupled Ricci flow in K\"ahler geometry, can also be viewed as an Abelian model for the Anomaly flow [PPZ1-2]. We believe that the $\kappa$-LYZ flow is a natural geometric flow in K\"ahler geometry.

%and then established the convergence of (\ref{1.7}) in some case.
%For $\kappa=1$ and $\lambda<0$ (e.g., for our flow (\ref{1.1}), under the assumption that a Mabuchi-type energy functional (see (4.2) in \cite{FGP1}) is bounded from below, they proved the convergence of %(\ref{1.7}) in some situations.

We were asked by the referee of \cite{LYZ} whether the nonnegativity of $\alpha$ in Theorem \ref{t1.2} can be removed.
Motivated by this question, in this paper, we consider the long time existence of the flow (\ref{1.7}) without assuming the condition (\ref{1.3}) and the nonnegativity of $\alpha$.
Our main theorem is as follows.

\begin{theorem}\label{t1.3} Let $(X,\omega)$ be a closed K\"ahler manifold of
complex dimension $n$. If the solution
$(\omega(t),\alpha(t))_{t\in[0,T)}$ of (\ref{1.7}), for $T <+\infty$, satisfies
$$
\sup_{X\times [0,T)}\left\{|{\rm Ric}(\omega(t))|_{\omega(t)}, |\alpha(t)|_{\omega(t)}\right\}\leq C_1
$$
for some $C_1>0$, then there exists $C_2>0$, depending only on $C_1, n, T, \kappa, \lambda$, as well as $(\omega, \alpha)$, such that $|{\rm Rm}(\omega(t))|_{\omega(t)} \leq C_2$.
\end{theorem}

As a consequence, we can generalize Theorem \ref{t1.2} above.

\begin{corollary}\label{c1.4}
Let
$(\omega(t),\alpha(t))$ be the solution to (\ref{1.7}) on $t\in[0,T_{\max})$ under the condition (\ref{1.8}) for some maximal
time $T_{\max}>0$.
 If $T_{\max}<+\infty$, then
\begin{equation}
\limsup_{t\to T_{\max}}\max_{X}\left\{|{\rm Ric}(\omega(t))|_{\omega(t)},
|\alpha(t)|_{\omega(t)}\right\}
=+\infty.\label{1.9}
\end{equation}
\end{corollary}

\begin{proof}
When $\kappa \not=1$, the statement actually follows from Theorem 1 in \cite{FGP1}.
For $\kappa =1, $ we
argue by contradiction. Suppose that $|{\rm Ric}(\omega(t))|_{\omega(t)}$
and $|\alpha(t)|_{\omega(t)}$ are uniformly bounded along the flow (\ref{1.7}) for $t \in [0, T_{\max})$. By
Theorem \ref{t1.3} above, $|{\rm Rm}(\omega(t))|_{\omega(t)}$ is also uniformly bounded. This contradicts with Theorem \ref{t1.1}.
\end{proof}

\begin{remark}\label{r1.5} $(1)$ Under the condition (\ref{1.8}), Theorem \ref{t1.3} is already obtained in \cite{FGP1}. In fact, stronger estimates are obtained there.

$(2)$ If $\kappa =1$, Corollary \ref{c1.4} holds without the
assumption (\ref{1.8}).
\end{remark}

%When $\kappa\neq1$ in (\ref{1.7}), the statement (\ref{1.8}) essentially follows
%from Theorem 1 in \cite{FGP1} under the cohomology condition (2.2) there.

%\begin{remark}\label{r1.5} Theorem \ref{t1.4} gives an affirmative answer to a question asked by referees on \cite{LYZ} whether
%the nonnegativity of $\alpha$ in Theorem \ref{t1.2} can be removed.
%\end{remark}

%\begin{remark}\label{r1.6} According to Theorem 1 in \cite{FGP}, the
%above Theorem \ref{t1.3} is also true for the $\kappa$-LYZ flow (\ref{1.7}) with $\kappa\neq1$ and $\kappa>0$. Consequently, Theorem \ref{t1.3} holds for all $\kappa>0$.
%\end{remark}

The strategy of proving Theorem \ref{t1.3} is to consider the local
curvature estimates for (\ref{1.7}), inspired from the work of Kotschwar, Munteanu, and Wang \cite{KMW}. Actually, the result in Theorem \ref{t1.3} holds for a class of
geometric flows introduced in (\ref{2.1}) -- (\ref{2.2}).

\begin{theorem}\label{t1.6}{\bf (see also Theorem \ref{t2.1})} Let $(M,g)$ be a closed Riemannian manifold of
dimension $n$ and $\alpha$ be a $2$-form on $M$. If the solution
$(g(t),\alpha(t))_{t\in[0,T]}$ of (\ref{2.1}) -- (\ref{2.2})
with initial data $(g(0),\alpha(0))=(g,\alpha)$, where $T$ is finite, satisfies that $|{\rm Ric}(g(t))|_{g(t)}$ and $|\alpha(t)|_{g(t)}$ are
uniformly bounded, then $|{\rm Rm}(g(t))|_{g(t)}$ is also
uniformly bounded.
\end{theorem}

We expect our study in the $\kappa$-LYZ flow would be beneficial to other coupled geometric flows.

%%%%%%%%%%%%%%%%%%%%%%%%%%%%%%%%%%%%%%%%%%%%%%%%%%%%%%%%%%%
%%%%%%%%%%%%%%%%%%%%%%%%%%%%%%%%%%%%%%%%%%%%%%%%%%%%%%%%%%%
\section{Basic equations}\label{section2}
%%%%%%%%%%%%%%%%%%%%%%%%%%%%%%%%%%%%%%%%%%%%%%%%%%%%%%%%%%%
%%%%%%%%%%%%%%%%%%%%%%%%%%%%%%%%%%%%%%%%%%%%%%%%%%%%%%%%%%%

In the following ``closed'' always means compact without boundary. We
always omit the time variable $t$ in local computations. In \cite{LYZ}, we have computed
\begin{eqnarray*}
\partial_{t}g_{i\overline{j}}&=&-R_{i\overline{j}}
-g_{i\overline{j}}+\alpha_{i\overline{j}},\\
\partial_{t}\alpha_{i\overline{j}}&=&\Delta_{\omega_{t}}
\alpha_{i\overline{j}}+g(t)^{-1}\ast g(t)^{-1}\ast{\rm Rm}_{g(t)}
\ast\alpha_{i\overline{j}}
-g(t)^{-1}\ast{\rm Ric}_{g(t)}\ast\alpha_{i\overline{j}}.
\end{eqnarray*}
Here $A\ast B$ can be any linear combination of tensor
products of tensors fields $A$ and $B$ formed by contractions on $A_{i_{1}
\cdots i_{k}}$ and $B_{j_{1}\cdots j_{\ell}}$ using the metric $g$.
\\

Motivated by (\ref{1.1}), we consider the system of equations on
a Riemannian manifold $(M, g)$ of dimension $n$:
\begin{eqnarray}
\partial_{t}g_{ij}&=&-2R_{ij}+(ag_{ij}+b\alpha_{ij}),\label{2.1}\\
\partial_{t}\alpha_{ij}&=&\Delta_{g(t)}\alpha_{ij}
+(g^{-1}\ast g^{-1}\ast{\rm Rm}\ast\alpha+g^{-1}\ast{\rm Ric}\ast\alpha).
\label{2.2}
\end{eqnarray}
Here $a, b$ are two given constants. Let $T_{\max}$ be the maximal time of the system (\ref{2.2}), and take $T\in(0, T_{\max})$.

\begin{theorem}\label{t2.1} Let $(M,g)$ be a closed Riemannian manifold of
dimension $n$ and $\alpha$ be a $2$-form on $M$. If the solution
$(g(t),\alpha(t))_{t\in[0,T]}$ of (\ref{2.1}) -- (\ref{2.2})
with initial data $(g(0),\alpha(0))=(g,\alpha)$, where $T$ is finite, satisfies that $|{\rm Ric}(g(t))|_{g(t)}$ and $|\alpha(t)|_{g(t)}$ are
uniformly bounded, then $|{\rm Rm}(g(t))|_{g(t)}$ is also
uniformly bounded.
\end{theorem}

In fact, the uniform bound for $|{\rm Rm}(g(t))|_{g(t)}$
in Theorem \ref{t2.1} depends only on $\max_{M}|{\rm Ric}(g(t))|_{g(t)}$, $\max_{M}|\alpha(t)|_{g(t)}$, $g$, $\alpha$, $T$, $n$, $a, b$. The explicit uniform bound will be obtained in the following proof.
\\

Let us temporarily denote
$$
\eta_{ij}:=-2R_{ij}+\epsilon_{ij}, \ \ \
\epsilon_{ij}:=ag_{ij}+b\alpha_{ij}, \ \ \
\beta_{ij}:=g^{-1}\ast g^{-1}\ast{\rm Rm}\ast\alpha
+g^{-1}\ast{\rm Ric}\ast\alpha.
$$
We have the following general equations
$$
\partial_{t}g_{ij}=\eta_{ij}=-2R_{ij}
+\epsilon_{ij}, \ \ \ \partial_{t}\alpha_{ij}
=\Delta_{g(t)}\alpha_{ij}+\beta_{ij}.
$$
Basic facts about evolutions of curvatures are (see for example \cite{CLN}):
\begin{eqnarray*}
\partial_{t}R^{\ell}_{ijk}&=&\frac{1}{2}g^{\ell p}(\nabla_{i}\nabla_{j}\eta_{kp}
+\nabla_{i}\nabla_{k}\eta_{jp}-\nabla_{i}\nabla_{p}\eta_{jk}\\
&&- \ \nabla_{j}\nabla_{i}\eta_{kp}-\nabla_{j}\nabla_{k}\eta_{ip}
+\nabla_{j}\nabla_{p}\eta_{ik}),\\
\partial_{t}R_{jk}&=&\frac{1}{2}g^{pq}
(\nabla_{q}\nabla_{j}\eta_{kp}+\nabla_{q}\nabla_{k}\eta_{jp}
-\nabla_{q}\nabla_{p}\eta_{jk}
-\nabla_{j}\nabla_{k}\eta_{qp}),\\
\partial_{t}R_{g(t)}&=&-\Delta_{g(t)}{\rm tr}_{g(t)}\eta(t)+{\rm div}_{g(t)}
\left({\rm div}_{g(t)}\eta(t)\right)-R_{ij}\eta^{ij},
\end{eqnarray*}
where $({\rm div}_{g(t)}\eta(t))_{j}:=\nabla^{i}\eta_{ij}$ is the divergence
of $\eta(t)$. The following
four lemmas can be found in \cite{CLN} (slightly modified).

\begin{lemma}\label{l2.1} For $\eta_{ij}=-2R_{ij}+(ag_{ij}+b\alpha_{ij})$, we have
\begin{equation}
\partial_{t}R^{\ell}_{ijk}=g^{-1}\ast\nabla^{2}{\rm Ric}
+g^{-1}\ast{\rm Rm}\ast{\rm Ric}+(g^{-1}\ast\nabla^{2}\alpha
+g^{-1}\ast{\rm Rm}\ast\alpha).\label{2.3}
\end{equation}
\end{lemma}

\begin{lemma}\label{l2.2} For $\eta_{ij}=-2R_{ij}+(ag_{ij}
+b\alpha_{ij})$, we have
\begin{equation}
\partial_{t}dV=\left(-R+\frac{an}{2}+\frac{b}{2}{\rm tr}\alpha\right)dV.
\label{2.4}
\end{equation}
\end{lemma}

Introduce the metric-dependent parabolic operator
$$
\Box\equiv\Box_{g(t)}:=\partial_{t}-\Delta_{g(t)}\equiv
\partial_{t}-\Delta.
$$

\begin{lemma}\label{l2.3} For $\eta_{ij}=-2R_{ij}+(ag_{ij}
+b\alpha_{ij})$, we have
\begin{equation}
\Box R=|{\rm Ric}|^{2}
+\left(-b\Delta{\rm tr}\alpha-aR-bR_{ij}\alpha^{ij}
+b\nabla^{i}\nabla^{j}\alpha_{ji}\right).
\label{2.5}
\end{equation}
\end{lemma}

\begin{lemma}\label{l2.4} One has
\begin{eqnarray}
\Box R_{ij}&=&-2R_{ik}R^{k}{}_{j}+2R_{pijq}R^{pq}\nonumber\\
&&+ \ \frac{b}{2}\left(-\Delta\alpha_{ij}+\nabla^{k}\nabla_{i}\alpha_{jk}
+\nabla^{k}\nabla_{j}\alpha_{ik}-\nabla_{i}\nabla_{j}
{\rm tr}\alpha\right).\label{2.6}
\end{eqnarray}
\end{lemma}

In particular,
\begin{eqnarray}
\Box|{\rm Ric}|^{2}&=&-2|\nabla{\rm Ric}|^{2}+2R^{ij}\Box R_{ij}
+2R_{ij}R^{i}{}_{q}\partial_{t}g^{jq}\nonumber\\
&=&-2|\nabla{\rm Ric}|^{2}+4R_{pijq}R^{ij}R^{pq}
-2a|{\rm Ric}|^{2}\label{2.7}\\
&&+ \ b R^{ij}\left(-\Delta\alpha_{ij}
+\nabla^{k}\nabla_{i}\alpha_{jk}
+\nabla^{k}\nabla_{j}\alpha_{ik}-\nabla_{i}\nabla_{j}{\rm tr}\alpha
-2R_{i}{}^{k}\alpha_{kj}\right)\nonumber
\end{eqnarray}
and
\begin{eqnarray}
|\nabla{\rm Ric}|^{2}&\leq&-\frac{1}{2}\Box|{\rm Ric}|^{2}
+C|{\rm Ric}|^{2}|{\rm Rm}|-a|{\rm Ric}|^{2}+C|b||{\rm Ric}|^{2}|\alpha|\nonumber\\
&&- \ \frac{b}{2}R^{ij}
\left(\Delta\alpha_{ij}-\nabla^{k}\nabla_{i}\alpha_{jk}-\nabla^{k}\nabla_{j}
\alpha_{ik}+\nabla_{i}\nabla_{j}{\rm tr}\alpha\right).\label{2.8}
\end{eqnarray}

\begin{lemma}\label{l2.5} One has
\begin{eqnarray}
\partial_{t}|{\rm Rm}|^{2}&=&\nabla^{2}{\rm Ric}\ast{\rm Rm}
+{\rm Ric}\ast{\rm Rm}\ast{\rm Rm}\nonumber\\
&&+ \ \left({\rm Rm}\ast{\rm Rm}+{\rm Rm}\ast\nabla^{2}\alpha
+{\rm Rm}\ast{\rm Rm}\ast\alpha\right).\label{2.9}
\end{eqnarray}
\end{lemma}

\begin{lemma}\label{l2.6} For $\partial_{t}g_{ij}=-2R_{ij}+a g_{ij}
+b\alpha_{ij}$, one has
\begin{eqnarray}
\Box R_{ijk\ell}&=&2(B_{ijk\ell}-B_{ij\ell k}+B_{ikj\ell}
-B_{i\ell jk})\nonumber\\
&&- \ (R_{ip}R^{p}_{jk\ell}+R_{jp}R_{ipk\ell}
+R_{k}{}^{p}R_{ijp\ell}+R_{\ell}{}^{p}R_{ijkp})\label{2.10}\\
&&+ \ C\nabla^{2}\alpha+{\rm Rm}\ast\alpha
+aR_{ijk\ell}+b\alpha_{\ell p}R^{p}_{ijk}.\nonumber
\end{eqnarray}
Here $B_{ijk\ell}=-R_{pijq}R^{p}{}_{k\ell}{}^{q}$.
\end{lemma}

In particular,
\begin{eqnarray}
|\nabla{\rm Rm}|^{2}&\leq&-\frac{1}{2}\Box|{\rm Rm}|^{2}
+C|{\rm Rm}|^{3}\nonumber\\
&&+ \ C\left({\rm Rm}\ast{\rm Rm}+{\rm Rm}\ast{\rm Rm}\ast\alpha
+{\rm Rm}\ast\nabla^{2}\alpha\right).\label{2.11}
\end{eqnarray}

\begin{lemma}\label{l2.7} One has
\begin{eqnarray}
(\partial_{t}-\Delta)|\alpha|^{2}&=&-2|\nabla\alpha|^{2}
+g^{\ast-2}\ast{\rm Rm}\ast\alpha^{\ast2}
+g^{-1}\ast{\rm Ric}\ast\alpha^{\ast2}\nonumber\\
&&+ \ g^{\ast-4}\ast g\ast{\rm Ric}\ast\alpha^{\ast2}
-2a|\alpha|^{2}+g^{\ast-4}\ast g\ast\alpha^{\ast3}.\label{2.12}
\end{eqnarray}
\end{lemma}

In particular we have
\begin{equation}
|\nabla\alpha|^{2}\leq-\frac{1}{2}(\partial_{t}-\Delta)|\alpha|^{2}
-a|\alpha|^{2}+C\left(|{\rm Rm}||\alpha|^{2}
+|{\rm Ric}||\alpha|^{2}+|\alpha|^{3}\right).\label{2.13}
\end{equation}

%%%%%%%%%%%%%%%%%%%%%%%%%%%%%%%%%%%%%%%%%%%%%%%%%%%%%%%%%%%%%%%%%%%%%%%%%%%%%%
%%%%%%%%%%%%%%%%%%%%%%%%%%%%%%%%%%%%%%%%%%%%%%%%%%%%%%%%%%%%%%%%%%%%%%%%%%%%%%
\section{The main inequality}\label{section3}
%%%%%%%%%%%%%%%%%%%%%%%%%%%%%%%%%%%%%%%%%%%%%%%%%%%%%%%%%%%%%%%%%%%%%%%%%%%%%%
%%%%%%%%%%%%%%%%%%%%%%%%%%%%%%%%%%%%%%%%%%%%%%%%%%%%%%%%%%%%%%%%%%%%%%%%%%%%%%

To prove Theorem \ref{t2.1}, we need a local curvature estimate on
$|{\rm Rm}(g(t))|$. For further study, we assume in this section that the manifold $M$
is complete, because of the localness. Actually, the
following inequality (\ref{3.29}) holds for any complete manifold $M$ provided
that the geodesic ball $B_{g(0)}(x_{0}, \rho/\sqrt{K})$ is compact.
\\

We consider the system (\ref{2.1}) -- (\ref{2.2}) and the condition
\begin{equation}
|{\rm Ric}(g(t))|_{g(t)}\leq K, \ \ \ |\alpha(t)|_{g(t)}
\leq L, \ \ \ t\in[0,T] \ \text{with} \ T\in(0,T_{\max}),\label{3.1}
\end{equation}
where $T_{\max}$ is the maximal time of the system (\ref{2.1}) -- (\ref{2.2}). Then all metrics are equivalent to $g(0)$. Moreover, from (\ref{2.4}), (\ref{2.8}), (\ref{2.9}), (\ref{2.11}), and (\ref{2.13}), we get
\begin{equation}
|\nabla{\rm Ric}|^{2}\leq-\frac{1}{2}\Box|{\rm Ric}|^{2}
+CK^{2}|{\rm Rm}|+CK^{2}L+{\rm Ric}\ast\nabla^{2}\alpha+{\rm Ric}\ast\nabla^{2}{\rm
tr}\alpha,\label{3.2}
\end{equation}
\begin{equation}
|\nabla{\rm Rm}|^{2}
\leq-\frac{1}{2}\Box|{\rm Rm}|^{2}
+C|{\rm Rm}|^{3}
+C|{\rm Rm}|^{2}+CL|{\rm Rm}|^{2}
+{\rm Rm}\ast\nabla^{2}\alpha,\label{3.3}
\end{equation}
and
\begin{eqnarray}
\partial_{t}|{\rm Rm}|^{2}&=&\nabla^{2}{\rm Ric}\ast{\rm Rm}
+{\rm Ric}\ast{\rm Rm}\ast{\rm Rm}\nonumber\\
&&+ \ \left({\rm Rm}\ast{\rm Rm}+{\rm Rm}\ast\nabla^{2}\alpha
+{\rm Rm}\ast{\rm Rm}\ast\alpha\right),\label{3.4}
\end{eqnarray}
\begin{equation}
|\nabla\alpha|^{2}\leq-\frac{1}{2}(\partial_{t}-\Delta)|\alpha|^{2}
+CL^{2}|{\rm Rm}|+CL^{2}(1+K+L),\label{3.5}
\end{equation}
\begin{equation}
\partial_{t}dV_{t}\leq C(1+K+L)dV_{t}.\label{3.6}
\end{equation}

Now we consider the quantity
\begin{equation}
\frac{d}{dt}\left(\int|{\rm Rm}|^{p}\phi^{2p}dV_{t}\right), \ \ \
\int:=\int_{M}, \ \ \ p\geq3,\label{3.7}
\end{equation}
where $\phi=\phi(x)$ is a cutoff function with compact support inside $M$.
Using the identity $
\partial_{t}|{\rm Rm}|^{p}=\partial_{t}(|{\rm Rm}|^{2})^{p/2}
=\frac{p}{2}|{\rm Rm}|^{p-2}\partial_{t}|{\rm Rm}|^{2}$ and (\ref{3.4}) we obtain
\begin{eqnarray*}
&&\frac{d}{dt}\left(\int|{\rm Rm}|^{p}\phi^{2p}dV_{t}\right) \ \ = \ \ \int\left(\partial_{t}|{\rm Rm}|^{p}\right)\phi^{2p}dV_{t}
+\int|{\rm Rm}|^{p}\phi^{2p}(\partial_{t}dV_{t})\\
&=&\int\frac{p}{2}|{\rm Rm}|^{p-2}\left(\partial_{t}
|{\rm Rm}|^{2}\right)\phi^{2p}dV_{t}
+\int|{\rm Rm}|^{p}\phi^{2p}(\partial_{t}dV_{t})\\
&=&\int|{\rm Rm}|^{p}\phi^{2p}(\partial_{t}dV_{t})
+\frac{p}{2}\int|{\rm Rm}|^{p-2}
\left(\nabla^{2}{\rm Ric}\ast{\rm Rm}+{\rm Ric}\ast{\rm Rm}
\ast{\rm Rm}\right.\\
&&+ \ \left. {\rm Rm}\ast{\rm Rm}
+{\rm Rm}\ast{\rm Rm}\ast\alpha+{\rm Rm}\ast\nabla^{2}\alpha
\right)\phi^{2p}dV_{t}\\
&\leq&Cp(1+K+L)\int|{\rm Rm}|^{p}\phi^{2p}dV_{t}
+Cp\int\left(\nabla^{2}{\rm Ric}\ast{\rm Rm}\right)|{\rm Rm}|^{p-2}\phi^{2p}
dV_{t}\\
&&+ \ Cp\int\left(\nabla^{2}\alpha\ast{\rm Rm}\right)|{\rm Rm}|^{p-2}
\phi^{2p}dV_{t}.
\end{eqnarray*}
The last two integrals can be simplified as follows.
\begin{eqnarray*}
&&C\int\left(\nabla^{2}{\rm Ric}\ast{\rm Rm}\right)|{\rm Rm}|^{p-2}
\phi^{2p}dV_{t} \ \ = \ \ C\int\nabla{\rm Ric}\ast
\left(\nabla{\rm Rm}\ast|{\rm Rm}|^{p-2}\ast\phi^{2p}\right.\\
&&+ \ \left.{\rm Rm}\ast\nabla|{\rm Rm}|^{p-2}\ast\phi^{2p}+{\rm Rm}\ast|{\rm Rm}|^{p-2}\ast\nabla\phi^{2p}\right)dV_{t}\\
&\leq&Cp\int|\nabla{\rm Ric}||\nabla{\rm Rm}||{\rm Rm}|^{p-2}\phi^{2p}dV_{t}
+Cp\int|\nabla{\rm Ric}||\nabla\phi||{\rm Rm}|^{p-1}
\phi^{2p-1}dV_{t}\\
&\leq&Cp\int\left(|\nabla{\rm Ric}||{\rm Rm}|^{\frac{p-1}{2}}
\phi^{p}\right)\left(|\nabla{\rm Rm}||{\rm Rm}|^{\frac{p-3}{2}}
\phi^{p}\right)dV_{t}\\
&&+ \ Cp\int\left(|\nabla{\rm Ric}||{\rm Rm}|^{\frac{p-1}{2}}
\phi^{p}\right)\left(|\nabla\phi||{\rm Rm}|^{\frac{p-1}{2}}
\phi^{p-1}\right)dV_{t}\\
&\leq&\frac{1}{pK}\int|\nabla{\rm Ric}|^{2}|{\rm Rm}|^{p-1}
\phi^{2p}dV_{t}+Cp^{3}K\int|\nabla{\rm Rm}|^{2}|{\rm Rm}|^{p-3}
\phi^{2p}dV_{t}\\
&&+ \ Cp^{3}K\int|\nabla\phi|^{2}|{\rm Rm}|^{p-1}\phi^{2p-2}dV_{t}.
\end{eqnarray*}
Similarly,
\begin{eqnarray*}
&&C\int\left(\nabla^{2}\alpha\ast{\rm Rm}\right)|{\rm Rm}|^{p-2}
\phi^{2p}dV_{t} \ \ = \ \ C\int\nabla\alpha\ast\left(
\nabla{\rm Rm}\ast|{\rm Rm}|^{p-2}\ast\phi^{2p}\right.\\
&&+ \ \left.{\rm Rm}\ast\nabla|{\rm Rm}|^{p-2}\ast\phi^{2p}
+{\rm Rm}\ast|{\rm Rm}|^{p-2}\ast\nabla\phi^{2p}\right)dV_{t}\\
&\leq &Cp\int|\nabla\alpha||\nabla{\rm Rm}||{\rm Rm}|^{p-2}
\phi^{2p}dV_{t}+Cp\int|\nabla\alpha||\nabla\phi||{\rm Rm}|^{p-1}
\phi^{2p-1}dV_{t}\\
&\leq&\frac{1}{pK}\int|\nabla\alpha|^{2}|{\rm Rm}|^{p-1}
\phi^{2p}dV_{t}+Cp^{3}K\int|\nabla{\rm Rm}|^{2}|{\rm Rm}|^{p-3}
\phi^{2p}dV_{t}\\
&&+ \ Cp^{3}K\int|\nabla\phi|^{2}|{\rm Rm}|^{p-1}\phi^{2p-2}dV_{t}.
\end{eqnarray*}
Hence
\begin{eqnarray*}
&&\frac{d}{dt}\left(\int|{\rm Rm}|^{p}\phi^{2p}dV_{t}\right) \ \
\leq \ \ \frac{1}{K}\int|\nabla{\rm Ric}|^{2}|{\rm Rm}|^{p-1}\phi^{2p}dV_{t}\\
&&+ \ Cp^{4}K\int|\nabla{\rm Rm}|^{2}|{\rm Rm}|^{p-3}\phi^{2p}dV_{t}
+Cp^{4}K\int|{\rm Rm}|^{p-1}|\nabla\phi|^{2}\phi^{2p-2}dV_{t}\\
&&+ \ Cp(1+K+L)\int|{\rm Rm}|^{p}\phi^{2p}dV_{t}
+\frac{1}{K}\int|\nabla\alpha|^{2}|{\rm Rm}|^{p-1}\phi^{2p}dV_{t}.
\end{eqnarray*}
Introduce five {\bf bad terms}, which involve derivatives of ${\rm Rm}$
and $\alpha$,
\begin{eqnarray*}
B_{1}&:=&\frac{1}{K}\int|\nabla{\rm Ric}|^{2}|{\rm Rm}|^{p-1}
\phi^{2p}dV_{t},\\
B_{2}&:=&\int|\nabla{\rm Rm}|^{2}|{\rm Rm}|^{p-3}\phi^{2p}dV_{t},\\
B_{3}&:=&\frac{1}{K}\int|\nabla\alpha|^{2}|{\rm Rm}|^{p-1}\phi^{2p}dV_{t},\\
B_{4}&:=&\frac{1}{K}\int|\nabla{\rm tr}\alpha|^{2}|{\rm Rm}|^{p-1}\phi^{2p}
dV_{t}\leq C B_{3},\\
B_{5}&:=&\int|\nabla\alpha|^{2}|{\rm Rm}|^{p-3}\phi^{2p}dV_{t}.
\end{eqnarray*}
Define also four {\bf good terms}
\begin{eqnarray*}
A_{1}&:=&\int|{\rm Rm}|^{p}\phi^{2p}dV_{t},\\
A_{2}&:=&\int|{\rm Rm}|^{p-1}|\nabla\phi|^{2}\phi^{2p-1}dV_{t},\\
A_{3}&:=&\int|{\rm Rm}|^{p-1}\phi^{2p}dV_{t},\\
A_{4}&:=&\int|{\rm Rm}|^{p-1}|\nabla\phi|^{2}\phi^{2p-2}
dV_{t}.
\end{eqnarray*}
Then the previous calculation can be written equivalently as
\begin{equation}
\frac{d}{dt} A_{1}\leq B_{1}+Cp^{4}K B_{2}+B_{3}+Cp^{4}KA_{4}
+Cp(1+K+L)A_{1}.\label{3.8}
\end{equation}

%%%%%%%%%%%%%%%%%%%%%%%%%%%%%%%%%%%%%%%%%%%%%%%%%%%%%%%%%%%
\subsection{Auxiliary lemmas}\label{subsection3.1}
%%%%%%%%%%%%%%%%%%%%%%%%%%%%%%%%%%%%%%%%%%%%%%%%%%%%%%%%%%%

We start with the following five lemmas.

\begin{lemma}\label{l3.1} We have
\begin{eqnarray}
B_{1}&\leq& Cp^{2}K B_{2}+Cp^{2}B_{3}+p^{2}CB_{4}\nonumber\\
&&+ \ Cp^{2}KA_{1}+Cp^{2}K(1+K+L)A_{2}+Cp^{2}KA_{4}\label{3.9}\\
&&- \ \frac{1}{2K}\frac{d}{dt}\left(\int|{\rm Ric}|^{2}|{\rm Rm}|^{p-1}
\phi^{2p}dV_{t}\right).\nonumber
\end{eqnarray}
\end{lemma}

\begin{proof} According to (\ref{3.2}) we obtain
\begin{eqnarray*}
&&B_{1} \ \ \leq \ \ \frac{1}{K}\int\left[\frac{1}{2}(\Delta
-\partial_{t})|{\rm Ric}|^{2}+CK^{2}|{\rm Rm}|+CK^{2}(1+L)
+{\rm Ric}\ast\nabla^{2}\alpha\right.\\
&&+ \ \left. \ {\rm Ric}\ast\nabla^{2}{\rm tr}\alpha
\right]|{\rm Rm}|^{p-1}\phi^{2p}dV_{t} \ \ = \ \
\frac{1}{2K}\int\left[(\Delta-\partial_{t})|{\rm Ric}|^{2}\right]|{\rm Rm}|^{p-1}
\phi^{2p}dV_{t}\\
&&+ \ CKA_{1}+CK(1+L)A_{2}+\frac{1}{K}
\int\left({\rm Ric}\ast\nabla^{2}\alpha
+{\rm Ric}\ast\nabla^{2}{\rm tr}\alpha\right)|{\rm Rm}|^{p-1}
\phi^{2p}dV_{t}.
\end{eqnarray*}
We first estimate the last two terms involving higher derivatives of
$\alpha$ as follows.
\begin{eqnarray*}
&&\frac{1}{K}\int\left({\rm Ric}\ast\nabla^{2}\alpha\right)
|{\rm Rm}|^{p-1}\phi^{2p}dV_{t} \ \ = \ \
\frac{1}{K}\int\nabla\alpha\ast\left(\nabla{\rm Ric}\ast|{\rm Rm}|^{p-1}
\ast\phi^{2p}\right.\\
&&+ \left. \ {\rm Ric}\ast\nabla|{\rm Rm}|^{p-1}\ast\phi^{2p}
+{\rm Ric}\ast|{\rm Rm}|^{p-1}\ast\nabla\phi^{2p}\right)dV_{t}\\
&\leq&\frac{C}{K}\int|\nabla\alpha||\nabla{\rm Ric}||{\rm Rm}|^{p-1}
\phi^{2p}dV_{t}+Cp\int|\nabla{\rm Rm}||\nabla\alpha||{\rm Rm}|^{p-2}
\phi^{2p}dV_{t}\\
&&+ \ Cp\int|\nabla\alpha||\nabla\phi||{\rm Rm}|^{p-1}\phi^{2p-1}dV_{t}\\
&\leq&\frac{C}{K}\int\left(|\nabla\alpha||{\rm Rm}|^{\frac{p-1}{2}}
\phi^{p}\right)\left(|\nabla{\rm Ric}||{\rm Rm}|^{\frac{p-1}{2}}
\phi^{p}\right)dV_{t}\\
&&+ \ Cp\int\left(|\nabla\alpha||{\rm Rm}|^{\frac{p-1}{2}}\phi^{p}
\right)\left(|\nabla{\rm Rm}||{\rm Rm}|^{\frac{p-3}{2}}
\phi^{p}\right)dV_{t}\\
&&+ \ Cp\int\left(|\nabla\alpha||{\rm Rm}|^{\frac{p-1}{2}}\phi^{p}
\right)\left(|\nabla\phi||{\rm Rm}|^{\frac{p-1}{2}}\phi^{p-1}\right)dV_{t}\\
&\leq&\frac{1}{100}B_{1}+CpB_{3}+CpKB_{2}+CpKA_{4}.
\end{eqnarray*}
The second one can be similarly computed:
$$
\frac{1}{K}\int\left({\rm Ric}\ast\nabla^{2}{\rm tr}\alpha
\right)|{\rm Rm}|^{p-1}\phi^{2p}dV_{t}
\leq\frac{1}{100}B_{1}+CpB_{4}+CpK B_{2}+CpKA_{4}.
$$
To deal with the term
$$
\frac{1}{2K}\int\left[(\Delta-\partial_{t})|{\rm Ric}|^{2}\right]
|{\rm Rm}|^{p-1}\phi^{2p}dV_{t},
$$
we calculate
\begin{eqnarray*}
&&\frac{1}{2K}\int\left[(\Delta-\partial_{t})
|{\rm Ric}|^{2}\right]|{\rm Rm}|^{p-1}\phi^{2p}dV_{t} \ \ = \ \
\frac{1}{2K}\int\left(\Delta|{\rm Ric}|^{2}\right)
|{\rm Rm}|^{p-1}\phi^{2p}dV_{t}\\
&&- \ \frac{1}{2K}\int\left[\partial_{t}\left(|{\rm Ric}|^{2}
|{\rm Rm}|^{p-1}\phi^{2p}dV_{t}\right)
-|{\rm Ric}|^{2}\left(\partial_{t}|{\rm Rm}|^{p-1}\right)\phi^{2p}dV_{t}\right.\\
&&- \left. \ |{\rm Ric}|^{2}
|{\rm Rm}|^{p-1}\phi^{2p}
\left(-R+\frac{a}{2}n+\frac{b}{2}{\rm tr}\alpha\right)dV_{t}\right]\\
&\leq&-\frac{1}{2K}\left[\int\left\langle\nabla|{\rm Ric}|^{2},
\nabla|{\rm Rm}|^{p-1}\right\rangle\phi^{2p}dV_{t}
+\int\left\langle\nabla|{\rm Ric}|^{2},
\nabla\phi^{2p}\right\rangle|{\rm Rm}|^{p-1}dV_{t} \right]\\
&&- \ \frac{1}{2K}\frac{d}{dt}
\left(\int|{\rm Ric}|^{2}|{\rm Rm}|^{p-1}\phi^{2p}dV_{t}\right)
+\frac{1}{2K}\int|{\rm Ric}|^{2}\left(\partial_{t}|{\rm Rm}|^{p-1}
\right)\phi^{2p}dV_{t}\\
&&+ \ CK(1+K+L)A_{2},
\end{eqnarray*}
where
\begin{eqnarray*}
&&-\frac{1}{2K}\int\left\langle\nabla|{\rm Ric}|^{2},
\nabla|{\rm Rm}|^{p-1}\right\rangle\phi^{2p}dV_{t}\\
&=&-\frac{1}{2K}\int\left(2{\rm Ric}\ast\nabla{\rm Ric}\ast\frac{p-1}{2}
|{\rm Rm}|^{p-3}\ast 2{\rm Rm}\ast\nabla{\rm Rm}\right)\phi^{2p}dV_{t}\\
&\leq&\frac{Cp}{K}\int|{\rm Ric}||\nabla{\rm Ric}||{\rm Rm}|^{p-2}
|\nabla{\rm Rm}|\phi^{2p}dV_{t}\\
&\leq&Cp\int\left(|\nabla{\rm Ric}||{\rm Rm}|^{\frac{p-1}{2}}\phi^{p}
\right)\left(|\nabla{\rm Rm}||{\rm Rm}|^{\frac{p-3}{2}}
\phi^{p}\right)dV_{t} \ \ \leq \ \ \frac{1}{100}B_{1}+Cp^{2}K B_{2}
\end{eqnarray*}
and
\begin{eqnarray*}
&&-\frac{1}{2K}\int\left\langle\nabla|{\rm Ric}|^{2},
\nabla\phi^{2p}\right\rangle|{\rm Rm}|^{p-1}dV_{t}\\
&=&-\frac{1}{2K}\int\left(2{\rm Ric}\ast\nabla{\rm Ric}\ast 2p\phi^{2p-1}
\nabla\phi\right)|{\rm Rm}|^{p-1}dV_{t}\\
&\leq& Cp\int|\nabla{\rm Ric}||\nabla\phi||{\rm Rm}|^{p-1}
\phi^{2p-1}dV_{t} \ \ \leq \ \ \frac{1}{100}B_{1}+Cp^{2}KA_{4}
\end{eqnarray*}
and, using (\ref{3.4}),
\begin{eqnarray*}
&&\frac{1}{2K}\int|{\rm Ric}|^{2}\left(\partial_{t}|{\rm Rm}|^{p-1}\right)
\phi^{2p}dV_{t}=\frac{p-1}{4K}\int|{\rm Ric}|^{2}
\left(|{\rm Rm}|^{p-3}\partial_{t}|{\rm Rm}|^{2}\right)\phi^{2p}dV_{t}\\
&=&\frac{Cp}{K}\int|{\rm Ric}|^{2}|{\rm Rm}|^{p-3}
\phi^{2p}\left[\nabla^{2}{\rm Ric}\ast{\rm Rm}
+{\rm Ric}\ast{\rm Rm}\ast{\rm Rm}\right.\\
&&+ \ \left. \ {\rm Rm}\ast{\rm Rm}+{\rm Rm}\ast{\rm Rm}\ast\alpha
+{\rm Rm}\ast\nabla^{2}\alpha\right]dV_{t}\\
&\leq&\frac{Cp}{K}\int\left(\nabla^{2}{\rm Ric}\ast{\rm Rm}\right)|{\rm Ric}|^{2}
|{\rm Rm}|^{p-3}\phi^{2p}
dV_{t}\\
&&+ \ CpK(1+K+L)A_{2}+\frac{Cp}{K}\int\left({\rm Rm}\ast\nabla^{2}
\alpha\right)|{\rm Ric}|^{2}|{\rm Rm}|^{p-3}\phi^{2p}dV_{t}.
\end{eqnarray*}
Here, under the same method,
\begin{eqnarray*}
&&\frac{Cp}{K}\int\left({\rm Rm}\ast\nabla^{2}\alpha\right)|{\rm Ric}|^{2}
|{\rm Rm}|^{p-3}\phi^{2p}dV_{t}\\
&=&
\frac{Cp}{K}\int\nabla\alpha\ast\left(\nabla{\rm Rm}\ast
|{\rm Ric}|^{2}|{\rm Rm}|^{p-3}\phi^{2p}+{\rm Rm}\ast\nabla|{\rm Ric}|^{2}\ast|{\rm Rm}|^{p-3}
\ast\phi^{2p}\right.\\
&&+ \ {\rm Rm}\ast|{\rm Ric}|^{2}\ast\nabla|{\rm Rm}|^{p-3}
\ast\phi^{2p}+\left. \ {\rm Rm}\ast|{\rm Ric}|^{2}\ast|{\rm Rm}|^{p-3}\ast\nabla\phi^{2p}
\right)dV_{t}\\
&\leq&Cp\int|\nabla\alpha||\nabla{\rm Rm}||{\rm Rm}|^{p-2}
\phi^{2p}dV_{t}+\frac{Cp}{K}\int|\nabla\alpha||\nabla{\rm Ric}|
|{\rm Rm}|^{p-1}\phi^{2p}dV_{t}\\
&&+ \ Cp\int|\nabla\alpha||\nabla\phi||{\rm Rm}|^{p-1}
\phi^{2p-1}dV_{t}\\
&\leq&\frac{1}{100}B_{1}+Cp^{2}B_{3}+Cp^{2}KB_{2}+Cp^{2}KA_{4}
\end{eqnarray*}
and
\begin{eqnarray*}
&&\frac{Cp}{K}\int|{\rm Ric}|^{2}|{\rm Rm}|^{p-3}\phi^{2p}
\left(\nabla^{2}{\rm Ric}\ast{\rm Rm}\right)dV_{t}\\
&=&\frac{Cp}{K}\nabla{\rm Ric}\ast\nabla\left(|{\rm Ric}|^{2}
|{\rm Rm}|^{p-3}\phi^{2p}\ast{\rm Rm}\right)dV_{t}\\
&=&\frac{Cp}{K}\int\nabla{\rm Ric}\ast\left({\rm Ric}\ast\nabla{\rm Ric}
\ast|{\rm Rm}|^{p-3}\phi^{2p}\ast{\rm Rm}+|{\rm Ric}|^{2}|{\rm Rm}|^{p-3}\ast\phi^{2p}\nabla{\rm Rm}\right.\\
&&+ \ \left.|{\rm Ric}|^{2}|{\rm Rm}^{p-3}\phi^{2p-1}\nabla\phi
\ast{\rm Rm}+|{\rm Ric}|^{2}|{\rm Rm}|^{p-4}
\nabla{\rm Rm}\ast\phi^{2p}{\rm Rm}\right)dV_{t}\\
&\leq&Cp\int|\nabla{\rm Ric}|^{2}
|{\rm Rm}|^{p-2}\phi^{2p}dV_{t}
+Cp\int|\nabla{\rm Ric}||\nabla{\rm Rm}||{\rm Ric}|
|{\rm Rm}|^{p-3}\phi^{2p}dV_{t}\\
&&+ \ \frac{Cp}{K}\int|\nabla{\rm Ric}||\nabla\phi||{\rm Ric}|^{2}
|{\rm Rm}|^{p-2}\phi^{2p-1}dV_{t}\\
&\leq&\frac{1}{100}B_{1}
+Cp^{2}KB_{2}+Cp^{2}KA_{4}.
\end{eqnarray*}
Therefore
\begin{eqnarray*}
&&\frac{1}{2K}\int|{\rm Ric}|^{2}\left(\partial_{t}|{\rm Rm}|^{p-1}\right)
\phi^{2p}dV_{t}\\
&\leq&\frac{1}{100}B_{1}+Cp^{2}K B_{2}+Cp^{2}K A_{4}\\
&&+ \ CpK(1+K+L)A_{2}+\frac{1}{100}B_{1}+Cp^{2}B_{3}+Cp^{2}KB_{2}
+Cp^{2}K A_{4}\\
&\leq&\frac{2}{100}B_{1}+Cp^{2}KB_{2}+Cp^{2}B_{3}
+CpK(1+K+L)A_{2}+Cp^{2}KA_{4}.
\end{eqnarray*}
Combining all estimates (\ref{3.9}) follows.
\end{proof}

\begin{lemma}\label{l3.2} We have
\begin{eqnarray}
B_{2}&\leq&Cp^{2}A_{1}+Cp^{2}(1+K+L)A_{2}+Cp^{2}A_{4}\nonumber\\
&&+ \ Cp^{2}B_{5}
-\frac{1}{p-1}\frac{d}{dt}\left(\int|{\rm Rm}|^{p-1}\phi^{2p}dV_{t}
\right).\label{3.10}
\end{eqnarray}
\end{lemma}

\begin{proof} From (\ref{3.3}) we obtain
\begin{eqnarray*}
B_{2}&=&\int|\nabla{\rm Rm}|^{2}|{\rm Rm}|^{p-3}\phi^{2p}
dV_{t}\\
&\leq&\int\left[\frac{1}{2}(\Delta-\partial_{t})
|{\rm Rm}|^{2}
+C|{\rm Rm}|^{3}+C|{\rm Rm}|^{2}\right.\\
&&+ \left. CL|{\rm Rm}|^{2}+{\rm Rm}\ast\nabla^{2}\alpha\right]|{\rm Rm}|^{p-3}
\phi^{2p}dV_{t}\\
&=&\frac{1}{2}\int\left(\Delta|{\rm Rm}|^{2}\right)
|{\rm Rm}|^{p-3}\phi^{2p}dV_{t}
-\frac{1}{2}\int\left(\partial_{t}|{\rm Rm}|^{2}\right)
|{\rm Rm}|^{p-3}\phi^{2p}dV_{t}\\
&&+ \ CA_{1}+C(1+L)A_{2}+\int\left({\rm Rm}\ast\nabla^{2}\alpha\right)
|{\rm Rm}|^{p-3}\phi^{2p}dV_{t}.
\end{eqnarray*}
Since
\begin{eqnarray*}
&&\int\left({\rm Rm}\ast\nabla^{2}\alpha\right)|{\rm Rm}|^{p-3}
\phi^{2p}dV_{t}\\
&=&\int\nabla\alpha\ast
\left(\nabla{\rm Rm}\ast|{\rm Rm}|^{p-3}\ast\phi^{2p}\right.\\
&&+ \left. {\rm Rm}\ast\nabla|{\rm Rm}|^{p-3}\ast\phi^{2p}+{\rm Rm}\ast|{\rm Rm}|^{p-3}\ast\nabla\phi^{2p}\right)
dV_{t}\\
&\leq&C\int|\nabla\alpha||\nabla{\rm Rm}||{\rm Rm}|^{p-3}
\phi^{2p}dV_{t}+Cp\int|\nabla\alpha||\nabla{\rm Rm}||{\rm Rm}|^{p-3}
\phi^{2p}dV_{t}\\
&&+ \ Cp\int|\nabla\alpha||\nabla\phi||{\rm Rm}|^{p-2}\phi^{2p-1}dV_{t}\\
&\leq&Cp\int\left(|\nabla\alpha||{\rm Rm}|^{\frac{p-3}{2}}
\phi^{p}\right)\left(|\nabla{\rm Rm}||{\rm Rm}|^{\frac{p-3}{2}}
\phi^{p}\right)dV_{t}\\
&&+ \ Cp\int\left(|\nabla\alpha||{\rm Rm}|^{\frac{p-3}{2}}\phi^{p}\right)
\left(|\nabla\phi||{\rm Rm}|^{\frac{p-1}{2}}\phi^{p-1}\right)dV_{t}\\
&\leq&\frac{1}{100}B_{2}+Cp^{2}B_{5}+CA_{4}
\end{eqnarray*}
and (since $p\geq3$)
\begin{eqnarray*}
&&\frac{1}{2}\int\left(\Delta|{\rm Rm}|^{2}\right)
|{\rm Rm}|^{p-3}\phi^{2p}dV_{t}\\
&=&-\frac{1}{2}\int\left\langle\nabla|{\rm Rm}|^{2}, \nabla
\left(|{\rm Rm}|^{p-3}\phi^{2p}\right)\right\rangle dV_{t}\\
&=&-\frac{1}{2}\int 2{\rm Rm}\ast\nabla{\rm Rm}
\ast\left(|{\rm Rm}|^{p-3}\ast 2p\phi^{2p-1}\ast\nabla\phi\right.\\
&&+ \ \left. \phi^{2p}\ast\frac{p-3}{2}|{\rm Rm}|^{p-5}\ast 2{\rm Rm}\nabla{\rm Rm}
\right)dV_{t}\\
&\leq&Cp\int|\nabla{\rm Rm}||\nabla\phi||{\rm Rm}|^{p-2}
\phi^{2p-1}dV_{t}
-(p-3)\int|\nabla{\rm Rm}|^{2}|{\rm Rm}|^{p-3}\phi^{2p}dV_{t}\\
&\leq&Cp\int|\nabla{\rm Rm}||\nabla\phi|
|{\rm Rm}|^{p-2}\phi^{2p-1}dV_{t}\\
&=&Cp\int\left(|\nabla{\rm Rm}||{\rm Rm}|^{\frac{p-3}{2}}
\phi^{p}\right)\left(
|\nabla\phi||{\rm Rm}|^{\frac{p-1}{2}}\phi^{p-1}\right)dV_{t}\\
&\leq&\frac{1}{100}B_{2}+Cp^{2}A_{4},
\end{eqnarray*}
we arrive at
\begin{eqnarray*}
B_{2}&\leq&\frac{1}{100}B_{2}+Cp^{2}A_{4}+CA_{1}+C(1+L)A_{2}\\
&&+ \ \frac{1}{100}B_{2}+Cp^{2}B_{5}+CA_{4}
-\frac{1}{2}\int\left(\partial_{t}|{\rm Rm}|^{2}\right)
|{\rm Rm}|^{p-3}\phi^{2p}dV_{t}.
\end{eqnarray*}
Therefore
$$
B_{2} \ \ \leq \ \ CA_{1}+C(1+L)A_{2}
+Cp^{2}A_{4}+Cp^{2}B_{5}-\frac{1}{2}\int\left(\partial_{t}|{\rm Rm}|^{2}
\right)|{\rm Rm}|^{p-3}\phi^{2p}dV_{t}.
$$
We now estimate the last integral. Direct computation shows
\begin{eqnarray*}
&&-\frac{1}{2}\int\left(\partial_{t}|{\rm Rm}|^{2}\right)
|{\rm Rm}|^{p-3}\phi^{2p}dV_{t}\\
&=&
-\frac{1}{2}\int\left[\partial_{t}\left(|{\rm Rm}|^{2}|{\rm Rm}|^{p-3}
\phi^{2p}dV_{t}\right)\right.\\
&&- \left. \ |{\rm Rm}|^{2}
\left(\partial_{t}|{\rm Rm}|^{p-3}\right)\phi^{2p}dV_{t}
-|{\rm Rm}|^{p-1}\phi^{2p}\left(\partial_{t}dV_{t}\right)
\right]\\
&=&-\frac{1}{2}\frac{d}{dt}\left(\int|{\rm Rm}|^{p-1}
\phi^{2p}dV_{t}\right)
+\frac{p-3}{4}\int|{\rm Rm}|^{p-3}
\left(\partial_{t}|{\rm Rm}|^{2}\right)\phi^{2p}dV_{t}\\
&&+ \ \frac{1}{2}\int|{\rm Rm}|^{p-1}\phi^{2p}
\left(-R+\frac{an}{2}+\frac{b}{2}{\rm tr}\alpha\right)dV_{t}.
\end{eqnarray*}
Thus
\begin{eqnarray*}
-\frac{1}{2}\int\left(\partial_{t}|{\rm Rm}|^{2}\right)
|{\rm Rm}|^{p-3}\phi^{2p}dV_{t}&\leq&\frac{C}{p-1}(1+K+L)A_{2}\\
&&- \ \frac{1}{p-1}\frac{d}{dt}\left(\int|{\rm Rm}|^{p-1}
\phi^{2p}dV_{t}\right).
\end{eqnarray*}
This implies (\ref{3.10}).
\end{proof}

\begin{lemma}\label{l3.3} We have
\begin{eqnarray}
B_{3}&\leq&C\frac{L^{2}}{K}A_{1}+Cp\frac{L^{2}}{K}(1+K+L)A_{2}
+Cp^{2}\frac{L^{2}}{K}A_{4}\nonumber\\
&&+ \ Cp^{2}\frac{L^{2}}{K}B_{2}+Cp^{2}\frac{L^{2}}{K}B_{5}
-\frac{1}{2K}\frac{d}{dt}\left(\int|\alpha|^{2}|{\rm Rm}|^{p-1}\phi^{2p}dV_{t}
\right).\label{3.11}
\end{eqnarray}
\end{lemma}

\begin{proof} Using (\ref{3.5}) implies
\begin{eqnarray*}
B_{3}&\leq&\frac{1}{K}\int|{\rm Rm}|^{p-1}
\phi^{2p}\left[-\frac{1}{2}(\partial_{t}-\Delta)|\alpha|^{2}
+CL^{2}|{\rm Rm}|+CL^{2}(1+K+L)\right]dV_{t}\\
&=&\frac{1}{2K}\int|{\rm Rm}|^{p-1}\phi^{2p}
(\Delta-\partial_{t})|\alpha|^{2}dV_{t}
+C\frac{L^{2}}{K}A_{1}+C\frac{L^{2}}{K}(1+K+L)A_{2}.
\end{eqnarray*}
The integral involving $\Delta|\alpha|^{2}$ can be simplified as
\begin{eqnarray*}
&&\frac{1}{2K}\int|{\rm Rm}|^{p-1}
\phi^{2p}\left(\Delta|\alpha|^{2}\right)dV_{t} \ \ = \ \ -\frac{1}{2K}\int\left\langle\nabla|\alpha|^{2},\nabla\left(|{\rm Rm}|^{p-1}
\phi^{2p}\right)\right\rangle dV_{t}\\
&=&-\frac{1}{2K}\int \alpha\ast\nabla\alpha\ast
\left[(p-1)|{\rm Rm}|^{p-3}\ast {\rm Rm}\ast\nabla{\rm Rm}\ast\phi^{2p}\right.\\
&&+ \ \left.2p|{\rm Rm}|^{p-1} \phi^{2p-1}\nabla\phi
\right]dV_{t}\\
&\leq&Cp\frac{L}{K}\int|\nabla\alpha||\nabla{\rm Rm}|
|{\rm Rm}|^{p-2}\phi^{2p}dV_{t}
+Cp\frac{L}{K}\int|\nabla\alpha||\nabla\phi||{\rm Rm}|^{p-1}
\phi^{2p-1}dV_{t}\\
&\leq &Cp\frac{L}{K}
\int\left(|\nabla{\rm Rm}||{\rm Rm}|^{\frac{p-3}{2}}
\phi^{p}\right)\left(|\nabla\alpha||{\rm Rm}|^{\frac{p-1}{2}}
\phi^{p}\right)dV_{t}\\
&&+ \ Cp\frac{L}{K}\int\left(|\nabla\phi||{\rm Rm}|^{\frac{p-1}{2}}
\phi^{p-1}\right)
\left(|\nabla\alpha||{\rm Rm}|^{\frac{p-1}{2}}\phi^{p}\right)dV_{t}\\
&\leq&\frac{2}{100}B_{3}+Cp\frac{L^{2}}{K}(B_{2}+A_{4}).
\end{eqnarray*}
The integral involving $\partial_{t}|\alpha|^{2}$ can be simplified as
\begin{eqnarray*}
&&\frac{1}{2K}\int|{\rm Rm}|^{p-1}\phi^{2p}
\left(-\partial_{t}|\alpha|^{2}\right)dV_{t} \ \ = \ \ -\frac{1}{2K}\int\left[\partial_{t}\left(|\alpha|^{2}
|{\rm Rm}|^{p-1}\phi^{2p}dV_{t}\right)\right.\\
&&- \ \left.|\alpha|^{2}\left(\partial_{t}|{\rm Rm}|^{p-1}\right)
\phi^{2p}dV_{t}
-|\alpha|^{2}|{\rm Rm}|^{p-1}\phi^{2p}(
\partial_{t}dV_{t})\right]\\
&=&-\frac{1}{2K}\frac{d}{dt}\left(\int|\alpha|^{2}
|{\rm Rm}|^{p-1}\phi^{2p}dV_{t}\right)
+\frac{p-1}{4K}\int|\alpha|^{2}|{\rm Rm}|^{p-3}
\left(\partial_{t}|{\rm Rm}|^{2}\right)\phi^{2p}dV_{t}\\
&&+ \ \frac{1}{2K}\int|\alpha|^{2}|{\rm Rm}|^{p-1}
\phi^{2p}\left(-R+\frac{an}{2}+\frac{b}{2}{\rm tr}\alpha\right)
dV_{t}\\
&\leq&-\frac{1}{2K}\frac{d}{dt}\left(\int|\alpha|^{2}
|{\rm Rm}|^{p-1}\phi^{2p}dV_{t}\right)
+C\frac{L^{2}}{K}(1+K+L)A_{2}\\
&&+ \ \frac{p-1}{4K}\int|\alpha|^{2}|{\rm Rm}|^{p-3}\left(\nabla^{2}{\rm Ric}\ast{\rm Rm}+{\rm Ric}\ast{\rm Rm}\ast{\rm Rm}\right.\\
&&+ \ \left.{\rm Rm}\ast{\rm Rm}+{\rm Rm}\ast{\rm Rm}\ast\alpha
+{\rm Rm}\ast\nabla^{2}\alpha\right)\phi^{2p}dV_{t}\\
&\leq&-\frac{1}{2K}\frac{d}{dt}\left(\int|\alpha|^{2}|{\rm Rm}|^{p-1}
\phi^{2p}dV_{t}\right)+Cp\frac{L^{2}}{K}(1+K+L)A_{2}\\
&&+ \ \frac{p-1}{4K}\int\left({\rm Rm}\ast\nabla^{2}{\rm Ric}
+{\rm Rm}\ast\nabla^{2}\alpha\right)|\alpha|^{2}|{\rm Rm}|^{p-3}\phi^{2p}dV_{t}
\end{eqnarray*}
using (\ref{3.4}). As before, we should estimate the integrals involving higher
derivatives. Because $p\geq3$,
\begin{eqnarray*}
&&\frac{p-1}{4K}\int|\alpha|^{2}|{\rm Rm}|^{p-3}
\left({\rm Rm}\ast\nabla^{2}{\rm Ric}\right)\phi^{2p}dV_{t}\\
&=&\frac{p-1}{4K}\int\nabla{\rm Ric}\ast\nabla
\left(|\alpha|^{2}|{\rm Rm}|^{p-3}\ast{\rm Rm}\ast \phi^{2p}\right)dV_{t}\\
&=&\frac{p-1}{4K}\int\nabla{\rm Ric}\ast\left(2\alpha\ast\nabla\alpha\ast|{\rm Rm}|^{p-3}
\ast{\rm Rm}\ast\phi^{2p}\right.\\
&&+ \ \left.|\alpha|^{2}
\ast\frac{p-3}{2}|{\rm Rm}|^{p-5}\ast 2{\rm Rm}\ast\nabla{\rm Rm}
\ast{\rm Rm}\ast\phi^{2p}\right.\\
&&+ \ \left.|\alpha|^{2}|{\rm Rm}|^{p-3}\ast\nabla{\rm Rm}\ast\phi^{2p}
+|\alpha|^{2}|{\rm Rm}|^{p-3}\ast{\rm Rm}\ast 2p\phi^{2p-1}
\nabla\phi\right)dV_{t}\\
&\leq&Cp\frac{L}{K}\int|\nabla{\rm Ric}||\nabla\alpha||{\rm Rm}|^{p-2}
\phi^{2p}dV_{t}\\
&&+ \ Cp^{2}\frac{L^{2}}{K}\int|\nabla{\rm Ric}||\nabla{\rm Rm}||{\rm Rm}|^{p-3}
\phi^{2p}dV_{t}\\
&&+ \ Cp^{2}\frac{L^{2}}{K}\int|\nabla{\rm Ric}||\nabla\phi||{\rm Rm}|^{p-2}
\phi^{2p-1}dV_{t}\\
&\leq&Cp\frac{L}{K}\int\left(|\nabla\alpha||{\rm Rm}|^{\frac{p-1}{2}}
\phi^{p}\right)\left(|\nabla{\rm Ric}||{\rm Rm}|^{\frac{p-3}{2}}
\phi^{p}\right)dV_{t}\\
&&+ \ Cp^{2}\frac{L^{2}}{K}\int\left(|\nabla{\rm Rm}||{\rm Rm}|^{\frac{p-3}{2}}
\phi^{p}\right)\left(|\nabla{\rm Ric}||{\rm Rm}|^{\frac{p-3}{2}}
\phi^{p}\right)dV_{t}\\
&&+ \ Cp^{2}\frac{L^{2}}{K}\int\left(|\nabla\phi||{\rm Rm}|^{\frac{p-1}{2}}
\phi^{p-1}\right)\left(|\nabla{\rm Ric}||{\rm Rm}|^{\frac{p-3}{2}}
\phi^{p}\right)dV_{t}\\
&\leq&\frac{1}{100}B_{3}+Cp^{2}\frac{L^{2}}{K}B_{2}+Cp^{2}\frac{L^{2}}{K}A_{4}.
\end{eqnarray*}
Similarly,
\begin{eqnarray*}
&&\frac{p-1}{4K}\int|\alpha|^{2}|{\rm Rm}|^{p-3}
\left({\rm Rm}\ast\nabla^{2}\alpha\right)\phi^{2p}dV_{t} \ \ = \ \
\frac{p-1}{4K}\int\nabla\alpha\ast\left(2\alpha\ast\nabla\alpha\ast\right.\\
&&\left.|{\rm Rm}|^{p-3}\ast{\rm Rm}\ast\phi^{2p}
+|\alpha|^{2}\ast\frac{p-3}{2}|{\rm Rm}|^{p-5}
\ast2{\rm Rm}\ast\nabla{\rm Rm}\ast{\rm Rm}\ast\phi^{2p}\right.\\
&&+ \ \left.|\alpha|^{2}|{\rm Rm}|^{p-3}\ast\nabla{\rm Rm}\ast\phi^{2p}
+|\alpha|^{2}|{\rm Rm}|^{p-3}\ast{\rm Rm}\ast 2p\phi^{2p-1}
\nabla\phi\right)dV_{t}\\
&\leq&Cp\frac{L}{K}\int|\nabla\alpha|^{2}
|{\rm Rm}|^{p-2}\phi^{2p}dV_{t}
+Cp^{2}\frac{L^{2}}{K}\int|\nabla\alpha||\nabla{\rm Rm}||{\rm Rm}|^{p-3}
\phi^{2p}dV_{t}\\
&&+ \ Cp^{2}\frac{L^{2}}{K}\int|\nabla\alpha||\nabla\phi||{\rm Rm}|^{p-2}
\phi^{2p-1}dV_{t}\\
&\leq&Cp\frac{L}{K}\int\left(|\nabla\alpha||{\rm Rm}|^{\frac{p-1}{2}}
\phi^{p}\right)\left(|\nabla\alpha||{\rm Rm}|^{\frac{p-3}{2}}
\phi^{p}\right)dV_{t}\\
&&+ \ Cp^{2}\frac{L^{2}}{K}\int\left(|\nabla\alpha|
|{\rm Rm}|^{\frac{p-3}{2}}\phi^{p}\right)
\left(|\nabla{\rm Rm}||{\rm Rm}|^{\frac{p-3}{2}}\phi^{p}\right)dV_{t}\\
&&+ \ Cp^{2}\frac{L^{2}}{K}\int\left(|\nabla\alpha||{\rm Rm}|^{\frac{p-3}{2}}
\phi^{p}\right)\left(|\nabla\phi||{\rm Rm}|^{\frac{p-1}{2}}
\phi^{p-1}\right)dV_{t}\\
&\leq&\frac{1}{100}B_{3}+Cp^{2}\frac{L^{2}}{K}B_{5}
+Cp^{2}\frac{L^{2}}{K}B_{2}+Cp^{2}\frac{L^{2}}{K}A_{4}.
\end{eqnarray*}
Those estimates yield (\ref{3.11}).
\end{proof}

\begin{lemma}\label{l3.4} We have
\begin{equation}
B_{4}\leq C B_{3}.\label{3.12}
\end{equation}
\end{lemma}

\begin{proof} It follows from the definitions.
\end{proof}

Finally, we estimate the term $B_{5}$.

\begin{lemma}\label{l3.5} We have
\begin{eqnarray}
B_{5}&\leq&-\frac{d}{dt}\left[\frac{1}{Cp^{4}L^{2}}\int|\alpha|^{2}|{\rm Rm}|^{p-1}
\phi^{2p}dV_{t}+\frac{1}{p^{2}(p-1)}\int|{\rm Rm}|^{p-1}
\phi^{2p}dV_{t}\right]\nonumber\\
&&- \ \frac{d}{dt}\left[\frac{C^{\frac{p-1}{2}}p^{2(p-3)}L^{p-3}}{K^{\frac{p-3}{2}}}
\int|\alpha|^{2}\phi^{2p}\right]+CA_{1}+C\left(1+K+L+\frac{K}{p^{4}L^{2}}\right)A_{2}
\nonumber\\
&&+\frac{C^{\frac{p-1}{2}}p^{2(p-3)}L^{p-1}}{K^{\frac{p-3}{2}}}\int|\nabla\phi|^{2}\phi^{2p-2}dV_{t}
+CA_{4}\label{3.13}\\
&&+ \ \frac{C^{\frac{(p-1)^{2}}{2(p-2)}}p^{2\frac{p^{2}-4p+5}{p-2}}(1+K+L)}{(K\wedge 1)^{\frac{1}{2}(p
-2+\frac{1}{p-2})}}(L\vee 1)^{p+2+\frac{1}{p-2}}
\int\phi^{2p}dV_{t}.\nonumber
\end{eqnarray}
Here
$$
K\wedge 1:=\min\{K,1\}, \ \ \ L\vee 1:=\max\{L,1\}.
$$
\end{lemma}

\begin{proof} For any $\eta>0$, we have
\begin{eqnarray*}
B_{5}&=&\int|\nabla\alpha|^{2}|{\rm Rm}|^{p-3}
\phi^{2p}dV_{t}\\
&\leq&\eta\int|\nabla\alpha|^{2}|{\rm Rm}|^{p-1}\phi^{2p}dV_{t}
+\frac{\left(\frac{p-3}{p-1}
\right)^{\frac{p-1}{2}}}{\eta^{\frac{p-3}{2}}}\int|\nabla\alpha|^{2}
\phi^{2p}dV_{t}\\
&=&\eta B_{3}+\frac{\left(\frac{p-3}{p-1}\right)^{\frac{p-1}{2}}}{\eta^{\frac{p-3}{2}}}
\int|\nabla\alpha|^{2}\phi^{2p}dV_{t}\\
&\leq&\eta B_{3}+\frac{1}{\eta^{\frac{p-3}{2}}}
\int|\nabla\alpha|^{2}\phi^{2p}dV_{t}
\end{eqnarray*}
because, $0\leq\frac{p-3}{p-1}<1$, and for any $\epsilon>0$,
\begin{eqnarray*}
|{\rm Rm}|^{p-3}\phi^{2p}&=&\left(\epsilon|{\rm Rm}|^{p-3}\phi^{\frac{p-3}{p-1}2p}
\cdot\frac{1}{\epsilon}\phi^{\frac{2}{p-1}2p}\right)\\
&\leq&\frac{p-3}{p-1}
\left(\epsilon|{\rm Rm}|^{p-3}\phi^{\frac{p-3}{p-1}2p}
\right)^{\frac{p-1}{p-3}}
+\frac{2}{p-1}\left(\frac{1}{\epsilon}\phi^{\frac{2}{p-1}2p}\right)^{\frac{p-1}{2}}\\
&\leq&\frac{p-3}{p-1}\epsilon^{\frac{p-1}{p-3}}
|{\rm Rm}|^{p-1}\phi^{2p}
+\frac{2}{p-1}\frac{1}{\epsilon^{\frac{p-1}{2}}}
\phi^{2p}.
\end{eqnarray*}
Using (\ref{2.13}) yields
\begin{eqnarray*}
B_{5}&\leq&\eta B_{3}
+\frac{1}{\eta^{\frac{p-3}{2}}}
\int\phi^{2p}\left[\frac{1}{2}(\Delta-\partial_{t})|\alpha|^{2}
+CL^{2}|{\rm Rm}|+CL^{2}(1+K+L)\right]dV_{t}\\
&\leq&\eta B_{3}
+\frac{C}{\eta^{\frac{p-3}{2}}}\int L^{2}(1+K+L)\phi^{2p}dV_{t}\\
&&+ \ \frac{C}{\eta^{\frac{p-3}{2}}2}
\int\phi^{2p}(\Delta-\partial_{t})
|\alpha|^{2}dV_{t}
+\frac{CL^{2}}{\eta^{\frac{p-3}{2}}}\int|{\rm Rm}|\phi^{2p}dV_{t}.
\end{eqnarray*}
To estimate the integral involving $|{\rm Rm}|\phi^{2p}$, we observe for $p\geq3$ that
\begin{eqnarray*}
|{\rm Rm}|\phi^{2p}&=&
\left(\epsilon|{\rm Rm}|\phi^{\frac{2p}{p-1}}\right)
\left(\frac{1}{\epsilon}\phi^{2p\frac{p-2}{p-1}}\right)\\
&\leq&\frac{1}{p-1}
\left(\epsilon|{\rm Rm}|\phi^{\frac{2p}{p-1}}\right)^{p-1}
+\frac{p-2}{p-1}\left(\frac{1}{\epsilon}
\phi^{2p\frac{p-2}{p-1}}\right)^{\frac{p-1}{p-2}}\\
&=&\frac{\epsilon^{p-1}}{p-1}
|{\rm Rm}|^{p-1}\phi^{2p}+\frac{p-2}{p-1}\frac{1}{\epsilon^{\frac{p-1}{p-2}}}
\phi^{2p}.
\end{eqnarray*}
Letting $\eta':=\epsilon^{p-1}/(p-1)$, it follows that
\begin{eqnarray}
|{\rm Rm}|\phi^{2p}&\leq&
\eta'|{\rm Rm}|^{p-1}\phi^{2p}
+\frac{p-2}{p-1}\frac{1}{(p-1)^{\frac{1}{p-2}}}
\frac{1}{\eta'{}^{\frac{1}{p-2}}}
\phi^{2p}\nonumber\\
&\leq&\eta'|{\rm Rm}|^{p-1}\phi^{2p}
+\frac{1}{\eta'{}^{\frac{1}{p-2}}}\phi^{2p}.\label{3.14}
\end{eqnarray}
Therefore
\begin{eqnarray*}
\frac{CL^{2}}{\eta^{\frac{p-3}{2}}}\int|{\rm Rm}|\phi^{2p}dV_{t}&
\leq&\frac{CL^{2}}{\eta^{\frac{p-3}{2}}}
\left(\eta'\int|{\rm Rm}|^{p-1}\phi^{2p}dV_{t}
+\frac{1}{\eta'{}^{\frac{1}{p-2}}}\int\phi^{2p}dV_{t}\right)\\
&\leq&\frac{CL^{2}\eta'}{\eta^{\frac{p-3}{2}}}
\int|{\rm Rm}|^{p-1}\phi^{2p}dV_{t}
+\frac{CL^{2}}{\eta^{\frac{p-3}{2}}\eta'{}^{\frac{1}{p-2}}}
\int\phi^{2p}dV_{t}.
\end{eqnarray*}
Choosing particularly $\eta':=\eta^{\frac{p-1}{2}}/CL^{2}$, it follows that
$$
\frac{CL^{2}}{\eta^{\frac{p-3}{2}}}
\int|{\rm Rm}|\phi^{2p}dV_{t} \ \ \leq \ \ \eta\int|{\rm Rm}|^{p-1}\phi^{2p}dV_{t}+\frac{(CL^{2})^{\frac{p-1}{p-2}}}{\eta^{\frac{p-3}{2}
+\frac{p-1}{2(p-2)}}}\int\phi^{2p}dV_{t}.
$$
Since $2<2\frac{p-1}{p-2}\leq4$ ($p\geq3$), it follows that
\begin{eqnarray*}
B_{5}&\leq&\eta B_{3}+\eta A_{2}+\frac{CL^{2}(1+K+L)}{\eta^{\frac{p-3}{2}}}
\int\phi^{2p}dV_{t}\\
&&+ \ \frac{(CL^{2})^{\frac{p-1}{p-2}}}{\eta^{\frac{p-3}{2}
+\frac{p-1}{2(p-2)}}}\int\phi^{2p}dV_{t}+\frac{C}{\eta^{\frac{p-3}{2}}2}
\int\phi^{2p}(\Delta-\partial_{t})|\alpha|^{2}dV_{t}\\
&\leq&\eta B_{3}+\eta A_{2}+\frac{CL^{2}(1+K+L)}{\eta^{\frac{p-3}{2}}}
\int\phi^{2p}dV_{t}+\frac{C(L\vee1)^{4}}{\eta^{\frac{p-3}{2}+\frac{p-1}{2(p-2)}}}
\int\phi^{2p}dV_{t}\\
&&+ \ \frac{C}{\eta^{\frac{p-3}{2}}2}\int\phi^{2p}
(\Delta-\partial_{t})|\alpha|^{2}dV_{t}.
\end{eqnarray*}
To estimate the last integral we start with
\begin{eqnarray*}
\int\left(\Delta|\alpha|^{2}\right)\phi^{2p}dV_{t}&=&
-\int\left\langle\Delta|\alpha|^{2},
\nabla\phi^{2p}\right\rangle dV_{t}\\
&=&-\int\left\langle 2\alpha\nabla\alpha, 2p\phi^{2p-1}\nabla
\phi\right\rangle dV_{t}\\
&=&-4p\int\alpha\phi^{2p-1}\langle\nabla\alpha,\nabla\phi\rangle dV_{t}\\
&\leq&Cp\int|\alpha||\nabla\alpha|\phi^{p}|\nabla\phi|\phi^{p-1}dV_{t}\\
&\leq&\epsilon\int|\nabla\alpha|^{2}\phi^{2p}dV_{t}
+\frac{Cp^{2}}{\epsilon}\int|\alpha|^{2}|\nabla\phi|^{2}
\phi^{2p-2}dV_{t}\\
&\leq&\frac{Cp^{2}L^{2}}{\epsilon}
\int|\nabla\phi|^{2}\phi^{2p-2}dV_{t}+\epsilon\int\phi^{2p}\bigg[\frac{1}{2}(\Delta
-\partial_{t})|\alpha|^{2}\\
&&+ \ CL^{2}|{\rm Rm}|+CL^{2}(1+K+L)\bigg]dV_{t}.
\end{eqnarray*}
Together with
\begin{eqnarray*}
\int\left(-\partial_{t}|\alpha|^{2}\right)
\phi^{2p}dV_{t}&=&
-\frac{d}{dt}\left(\int|\alpha|^{2}\phi^{2p}dV_{t}\right)
+\int|\alpha|^{2}\phi^{2p}\partial_{t}dV_{t}\\
&\leq&C(1+K+L)\int|\alpha|^{2}\phi^{2p}dV_{t}
-\frac{d}{dt}\left[\int|\alpha|^{2}\phi^{2p}dV_{t}\right],
\end{eqnarray*}
we arrive at
\begin{eqnarray*}
\int[(\Delta-\partial_{t})|\alpha|^{2}]
\phi^{2p}dV_{t}&\leq&\frac{\epsilon}{2}\int\phi^{2p}[(\Delta-\partial_{t})|\alpha|^{2}]dV_{t}\\
&&+ \ \frac{Cp^{2}L^{2}}{\epsilon}\int|\nabla\phi|^{2}\phi^{2p-2}dV_{t}\\
&&+ \ C(1+\epsilon)L^{2}(1+K+L)\int\phi^{2p}dV_{t}\\
&&+ \ C\epsilon L^{2}\int\phi^{2p}|{\rm Rm}|dV_{t}\\
&&- \ \frac{d}{dt}\left[\int|\alpha|^{2}\phi^{2p}dV_{t}\right].
\end{eqnarray*}
Therefore (taking $\epsilon=1$)
\begin{eqnarray*}
&&\int\left[(\Delta-\partial_{t})|\alpha|^{2}\right]
\phi^{2p}dV_{t}\\
&\leq&Cp^{2}L^{2}\int|\nabla\phi|^{2}\phi^{2p-2}
dV_{t}
+CL^{2}(1+K+L)\int\phi^{2p}dV_{t}\\
&&+ \ CL^{2}\int|{\rm Rm}|\phi^{2p}dV_{t}
-\frac{d}{dt}\left[2\int|\alpha|^{2}\phi^{2p}dV_{t}\right]\\
&\leq&Cp^{2}L^{2}\int|\nabla\phi|^{2}\phi^{2p-2}dV_{t}
+CL^{2}(1+K+L)\int\phi^{2p}dV_{t}\\
&&+ \ CL^{2}\overline{\eta}A_{2}+\frac{CL^{2}}{\overline{\eta}{}^{\frac{1}{p-2}}}
\int\phi^{2p}dV_{t}
-\frac{d}{dt}\left[2\int|\alpha|^{2}\phi^{2p}dV_{t}\right]
\end{eqnarray*}
using (\ref{3.14}), where $\overline{\eta}$ is any positive number. Consequently
\begin{eqnarray*}
B_{5}&\leq&\eta B_{3}+\eta A_{2}
+\left[\frac{CL^{2}(1+K+L)}{\eta^{\frac{p-3}{2}}}
+\frac{C(L\vee1)^{4}}{\eta^{\frac{p-3}{2}+\frac{p-1}{2(p-2)}}}
\right]\int\phi^{2p}dV_{t}+\frac{CL^{2}\overline{\eta}}{\eta^{\frac{p-3}{2}}}
A_{2}\\
&&+ \ \frac{Cp^{2}L^{2}}{\eta^{\frac{p-3}{2}}}
\int|\nabla\phi|^{2}\phi^{2p-2}dV_{t}+\frac{CL^{2}}{\eta^{\frac{p-3}{2}}\overline{\eta}{}^{\frac{1}{p-2}}}
\int\phi^{2p}dV_{t}\\
&&- \ \frac{d}{dt}
\left[\frac{C}{\eta^{\frac{p-3}{2}}}\int|\alpha|^{2}\phi^{2p}dV_{t}\right].
\end{eqnarray*}
If we take $CL^{2}\overline{\eta}/\eta^{\frac{p-3}{2}}=\eta$, then $
\overline{\eta}=\eta^{\frac{p-1}{2}}/CL^{2}$ and
$$
\frac{CL^{2}}{\eta^{\frac{p-3}{2}}\overline{\eta}{}^{\frac{1}{p-2}}}
=\frac{(CL^{2})^{\frac{p-1}{p-2}}}{\eta^{\frac{p-3}{2}+\frac{p-1}{2(p-2)}}}.
$$
Hence
\begin{eqnarray}
B_{5}&\leq&\eta B_{3}+2\eta A_{2}-\frac{d}{dt}\left[\frac{C}{\eta^{\frac{p-3}{2}}}
\int|\alpha|^{2}\phi^{2p}dV_{t}\right]\nonumber\\
&&+ \ \frac{Cp^{2}L^{2}}{\eta^{\frac{p-3}{2}}}
\int|\nabla\phi|^{2}\phi^{2p-2}dV_{t}\label{3.15}\\
&&+ \ \left[\frac{CL^{2}(1+K+L)}{\eta^{\frac{p-3}{2}}}
+\frac{C(L\vee1)^{4}}{\eta^{\frac{p-3}{2}+\frac{p-1}{2(p-2)}}}
\right]\int\phi^{2p}dV_{t}.\nonumber
\end{eqnarray}
From (\ref{3.10}) and (\ref{3.11}) we obtain
\begin{eqnarray}
B_{3}&\leq&-\frac{1}{2K}\frac{d}{dt}\left[\int|\alpha|^{2}
|{\rm Rm}|^{p-1}\phi^{2p}dV_{t}\right]\nonumber\\
&&- \ \frac{p^{2}}{p-1}\frac{CL^{2}}{K}\frac{d}{dt}
\left[\int|{\rm Rm}|^{p-1}\phi^{2p}dV_{t}\right]\label{3.16}\\
&&+ \ \frac{Cp^{4}L^{2}}{K}A_{1}+\frac{Cp^{4}L^{2}}{K}(1+K+L)A_{2}
+\frac{Cp^{4}L^{2}}{K}A_{4}+\frac{Cp^{4}L^{2}}{K}B_{5}.\nonumber
\end{eqnarray}
Plugging (\ref{3.16}) into (\ref{3.15}), it follows that
\begin{eqnarray*}
B_{5}&\leq&-\frac{\eta}{2K}\frac{d}{dt}\left[\int|\alpha|^{2}
|{\rm Rm}|^{p-1}\phi^{2p}dV_{t}\right]-\frac{p^{2}}{p-1}\frac{C\eta L^{2}}{K}
\frac{d}{dt}\left[\int|{\rm Rm}|^{p-1}\phi^{2p}dV_{t}\right]\\
&&- \ \frac{d}{dt}
\left[\frac{C}{\eta^{\frac{p-3}{2}}}
\int|\alpha|^{2}\phi^{2p}dV_{t}\right]\\
&&+ \ \frac{Cp^{4}\eta L^{2}}{K}\left[A_{1}+(1+K+L)A_{2}+A_{4}\right]
+\frac{Cp^{4}\eta L^{2}}{K}B_{5}\\
&&+ \ 2\eta A_{2}+\frac{CL^{2}}{\eta^{\frac{p-3}{2}}}\int|\nabla\phi|^{2}
\phi^{2p-2}dV_{t}\\
&&+ \ \left[\frac{Cp^{2}L^{2}(1+K+L)}{\eta^{\frac{p-3}{2}}}
+\frac{C(L\vee1)^{4}}{\eta^{\frac{p-3}{2}+\frac{p-1}{2(p-2)}}}
\right]\int\phi^{2p}dV_{t}.
\end{eqnarray*}
Taking $\eta:=K/2Cp^{4}L^{2}$, we have
$$
\frac{Cp^{4}\eta L^{2}}{K}=\frac{1}{2}, \ \ \
\frac{\eta}{2K}=\frac{1}{4Cp^{4}L^{2}},
$$
and then ($C\geq1$)
\begin{eqnarray*}
B_{5}&\leq&-\frac{d}{dt}\left[\frac{1}{Cp^{4}L^{2}}
\int|\alpha|^{2}|{\rm Rm}|^{p-1}\phi^{2p}dV_{t}
+\frac{1}{p^{2}(p-1)}\int|{\rm Rm}|^{p-1}\phi^{2p}dV_{t}\right]\\
&&- \ \frac{d}{dt}\left[\frac{C^{\frac{p-1}{2}}p^{2(p-3)}L^{p-3}}{K^{\frac{p-3}{2}}}
\int|\alpha|^{2}\phi^{2p}\right]+\left[A_{1}+(1+K+L)A_{2}+A_{4}\right]\\
&&+ \ \frac{4K}{Cp^{4}L^{2}}A_{2}+\frac{C^{\frac{p-1}{2}}p^{2(p-3)}L^{p-1}}{K^{\frac{p-3}{2}}}\int|\nabla\phi|^{2}\phi^{2p-2}
dV_{t}+\int\phi^{2p}dV_{t}\\
&&\cdot \ \left[\frac{C^{\frac{p-1}{2}}p^{2(p-2)}L^{p-1}(1+K+L)}{K^{\frac{p-3}{2}}}
+\frac{C^{\frac{(p-1)^{2}}{2(p-2)}}p^{2\frac{p^{2}-4p+5}{p-2}}
(L\vee1)^{\frac{p^{2}-3}{p-2}}}{K^{\frac{p-3}{2}+\frac{p-1}{2(p-2)}}}\right]
\\
&\leq&-\frac{d}{dt}\left[\frac{1}{Cp^{4}L^{2}}\int|\alpha|^{2}|{\rm Rm}|^{p-1}
\phi^{2p}dV_{t}+\frac{1}{p^{2}(p-1)}\int|{\rm Rm}|^{p-1}\phi^{2p}dV_{t}\right]\\
&&- \ \frac{d}{dt}\left[\frac{C^{\frac{p-1}{2}}p^{2(p-3)}L^{p-3}}{K^{\frac{p-3}{2}}}
\int|\alpha|^{2}\phi^{2p}\right]+
CA_{1}+C\left(1+K+L+\frac{K}{p^{4}L^{2}}\right)A_{2}\\
&&+ \ CA_{4}+ \frac{C^{\frac{p-1}{2}}p^{2(p-3)}L^{p-1}}{K^{\frac{p-3}{2}}}\int|\nabla\phi|^{2}\phi^{2p-2}dV_{t}\\
&&+ \ \frac{C^{\frac{(p-1)^{2}}{2(p-2)}}p^{2\frac{p^{2}-4p+5}{p-2}}(1+K+L)}{(K\wedge 1)^{\frac{1}{2}(p
-2+\frac{1}{p-2})}}(L\vee 1)^{p+2+\frac{1}{p-2}}
\int\phi^{2p}dV_{t}.
\end{eqnarray*}
Here $K\wedge 1:=\min\{K,1\}$ and $L\vee 1:=\max\{L,1\}$.
\end{proof}

For convenience, let
\begin{equation}
\mathfrak{a}_{1}:=1+K+L+\frac{K}{p^{4}L^{2}}, \ \ \
\mathfrak{a}_{2}:=\frac{L^{p-1}}{K^{\frac{p-3}{2}}}, \ \ \
\mathfrak{a}_{3}:=\frac{(1+K+L)(L\vee 1)^{p+2+\frac{1}{p-2}}}{(K\wedge 1)^{\frac{1}{2}(p
-2+\frac{1}{p-2})}}.\label{3.17}
\end{equation}
Then
\begin{eqnarray*}
B_{5}&\leq &-\frac{d}{dt}\left[\frac{1}{Cp^{4}L^{2}}\int|\alpha|^{2}
|{\rm Rm}|^{p-1}\phi^{2p}dV_{t}+\frac{1}{p^{2}(p-1)}\int|{\rm Rm}|^{p-1}\phi^{2p}dV_{t}\right]\\
&&- \ \frac{d}{dt}\left[\frac{C^{\frac{p-1}{2}}p^{2(p-3)}\mathfrak{a}_{2}}{L^{2}}
\int|\alpha|^{2}\phi^{2p}\right]+ \ CA_{1}+C\mathfrak{a}_{1}A_{2}+CA_{4}\\
&&+ \ C^{\frac{p-1}{2}}p^{2(p-3)}\mathfrak{a}_{2}\int|\nabla\phi|^{2}
\phi^{2p-2}dV_{t}+C^{\frac{(p-1)^{2}}{2(p-2)}}p^{2\frac{p^{2}-4p+5}{p-2}}\mathfrak{a}_{3}\int\phi^{2p}dV_{t}.
\end{eqnarray*}
Hence from (\ref{3.16}) we have
\begin{eqnarray*}
B_{3}&\leq&-\frac{d}{dt}\left[\frac{C}{K}\int|\alpha|^{2}|{\rm Rm}|^{p-1}
\phi^{2p}dV_{t}+\frac{p^{2}}{p-1}\frac{CL^{2}}{K}\int|{\rm Rm}|^{p-1}\phi^{2p}dV_{t}\right]\\
&&- \ \frac{d}{dt}\left[\frac{C^{\frac{p+1}{2}}p^{2(p-1)}\mathfrak{a}_{2}}{K}
\int|\alpha|^{2}\phi^{2p}\right]
+\frac{CL^{2}}{K}A_{1}+\frac{Cp^{4}L^{2}}{K}\mathfrak{a}_{1}A_{2}
+\frac{Cp^{4}L^{2}}{K}A_{4}\\
&&+ \ \frac{C^{\frac{p+1}{2}}p^{2(p-1)}L^{2}}{K}\mathfrak{a}_{2}\int|\nabla\phi|^{2}
\phi^{2p-2}dV_{t}+\frac{C^{\frac{p^{2}-3}{2(p-2)}}p^{2\frac{(p-1)^{2}}{p-2}}L^{2}}{K}\mathfrak{a}_{3}\int\phi^{2p}dV_{t},
\end{eqnarray*}
and from (\ref{3.10}) we have
\begin{eqnarray*}
B_{2}&\leq&-\frac{d}{dt}\left[\frac{C}{p^{2}L^{2}}\int|\alpha|^{2}
|{\rm Rm}|^{p-1}\phi^{2p}dV_{t}+C\frac{p^{2}+1}{p-1}\int|{\rm Rm}|^{p-1}\phi^{2p}dV_{t}\right]\\
&&- \ \frac{d}{dt}\left[\frac{C^{\frac{p+1}{2}}p^{2(p-2)}\mathfrak{a_{2}}}{L^{2}}
\int|\alpha|^{2}\phi^{2p}\right]+ Cp^{2}A_{1}+Cp^{2}\mathfrak{a}_{1}A_{2}+Cp^{2}A_{4}\\
&&+ \ C^{\frac{p+1}{2}}p^{2(p-2)}\mathfrak{a}_{2}\int|\nabla\phi|^{2}
\phi^{2p-2}dV_{t}+C^{\frac{p^{2}-3}{2(p-2)}}
p^{2\frac{p^{2}-3p+3}{p-2}}\mathfrak{a}_{3}\int\phi^{2p}dV_{t}.
\end{eqnarray*}
From (\ref{3.9}) we see that
\begin{eqnarray*}
B_{1}&\leq&-\frac{d}{dt}\left[\frac{1}{2K}\int|{\rm Ric}|^{2}|{\rm Rm}|^{p-1}
\phi^{2p}dV_{t}+\frac{Cp^{4}(K^{2}+L^{2})}{K(p-1)}\int|{\rm Rm}|^{p-1}
\phi^{2p}dV_{t}\right.\\
&&+ \ \left.\frac{Cp^{2}(K^{2}+L^{2})}{KL^{2}}\int|\alpha|^{2}
|{\rm Rm}|^{p-1}\phi^{2p}dV_{t}+\frac{C^{\frac{p+3}{2}}p^{2p}\mathfrak{a}_{2}}{K}
\int|\alpha|^{2}\phi^{2p}dV_{t}\right]\\
&&+ \ \frac{Cp^{4}(K^{2}+L^{2})}{K}A_{1}
+\frac{Cp^{6}(K^{2}+L^{2})}{K}\mathfrak{a}_{1}A_{2}
+ \frac{Cp^{6}(K^{2}+L^{2})}{K}
A_{4}\\
&&+ \ \frac{C^{\frac{p+3}{2}}p^{2p}(K^{2}+L^{2})}{K}\mathfrak{a}_{2}\int|\nabla\phi|^{2}
\phi^{2p-2}dV_{t}\\
&&+ \ \frac{C^{\frac{p^{2}+2p-7}{2(p-2)}}p^{2\frac{p^{2}-p-1}{p-2}}(K^{2}+L^{2})}{K}\mathfrak{a}_{3}
\int\phi^{2p}dV_{t}.
\end{eqnarray*}
Finally, (\ref{3.8}) yields
\begin{eqnarray*}
A'_{1}&\leq &-\frac{d}{dt}\left[\frac{1}{2K}\int|{\rm Ric}|^{2}|{\rm Rm}|^{p-1}
\phi^{2p}dV_{t}+\frac{Cp^{6}(K^{2}+L^{2})}{K(p-1)}\int|{\rm Rm}|^{p-1}\phi^{2p}dV_{t}\right.\\
&&+ \ \frac{Cp^{2}(K^{2}+L^{2})}{KL^{2}}\int|\alpha|^{2}
|{\rm Rm}|^{p-1}\phi^{2p}dV_{t}\\
&&+ \ \left.\frac{C^{\frac{p+3}{2}}p^{2p}(K^{2}+L^{2})}{KL^{2}}
\int|\alpha|^{2}\phi^{2p}dV_{t}\right]\\
&&+ \ \frac{Cp^{6}(K^{2}+L^{2})}{K}\mathfrak{a}_{1}A_{1}
+\frac{Cp^{6}(K^{2}+L^{2})}{K}\mathfrak{a}_{1}A_{2}+\frac{Cp^{6}(K^{2}+L^{2})}{K}A_{4}\\
&&+ \ \frac{C^{\frac{p+3}{2}}p^{2p}(K^{2}+L^{2})}{K}\mathfrak{a}_{2}\int|\nabla\phi|^{2}
\phi^{2p-2}dV_{t}+\frac{C^{\frac{p^{2}+2p-7}{2(p-2)}}p^{2\frac{p^{2}-p-1}{p-2}}(K^{2}+L^{2})}{K}\mathfrak{a}_{3}
\int\phi^{2p}dV_{t}.
\end{eqnarray*}
Because
$$
\frac{p^{2}+2p-7}{2(p-2)}\leq\frac{p}{2}, \ \ \ \frac{p^{2}-p-1}{p-2}\leq 2p, \ \ \ p\geq3,
$$
we arrive at
\begin{eqnarray*}
A'_{1}&\leq&-\frac{d}{dt}\left[\frac{1}{2K}\int|{\rm Ric}|^{2}|{\rm Rm}|^{p-1}
\phi^{2p}dV_{t}+\frac{Cp^{6}(K^{2}+L^{2})}{K(p-1)}\int|{\rm Rm}|^{p-1}\phi^{2p}dV_{t}\right.\\
&&+ \ \left.\frac{Cp^{2}(K^{2}+L^{2})}{KL^{2}}\int|\alpha|^{2}
|{\rm Rm}|^{p-1}\phi^{2p}dV_{t}+\frac{C^{\frac{p+3}{2}}p^{2p}(K^{2}+L^{2})}{KL^{2}}
\int|\alpha|^{2}\phi^{2p}dV_{t}\right]\\
&&+ \ \frac{Cp^{6}(K^{2}+L^{2})}{K}\mathfrak{a}_{1}A_{1}
+\frac{Cp^{6}(K^{2}+L^{2})}{K}\mathfrak{a}_{1}A_{2}+\frac{Cp^{6}(K^{2}+L^{2})}{K}A_{4}\\
&&
+ \ \frac{C^{\frac{p+3}{2}}p^{2p}(K^{2}+L^{2})}{K}\mathfrak{a}_{2}\int|\nabla\phi|^{2}
\phi^{2p-2}dV_{t}+\frac{C^{\frac{p}{2}}p^{p}(K^{2}+L^{2})}{K}\mathfrak{a}_{3}
\int\phi^{2p}dV_{t}.
\end{eqnarray*}

%%%%%%%%%%%%%%%%%%%%%%%%%%%%%%%%%%%%%%%%%%%%%%%%%%%%%%%%%%%
\subsection{Local curvature estimates}\label{subsection3.2}
%%%%%%%%%%%%%%%%%%%%%%%%%%%%%%%%%%%%%%%%%%%%%%%%%%%%%%%%%%%

To prove Theorem \ref{t2.1}, we introduce the following quantity
\begin{eqnarray}
U(t)&:=&\int|{\rm Rm}|^{p}\phi^{2p}dV_{t}
+\frac{1}{2K}\int|{\rm Ric}|^{2}|{\rm Rm}|^{p-1}
\phi^{2p}dV_{t}\nonumber\\
&&+ \ \frac{Cp^{6}(K^{2}+L^{2})}{K(p-1)}\int|{\rm Rm}|^{p-1}
\phi^{2p}dV_{t}\nonumber\\
&&+ \ \frac{Cp^{2}(K^{2}+L^{2})}{KL^{2}}\int|\alpha|^{2}
|{\rm Rm}|^{p-1}\phi^{2p}dV_{t}\label{3.18}\\
&&+ \ \frac{C^{\frac{p+3}{2}}p^{2p}(K^{2}+L^{2})}{KL^{2}}
\int|\alpha|^{2}\phi^{2p}dV_{t}.\nonumber
\end{eqnarray}
Then
\begin{eqnarray}
U'&\leq &\frac{Cp^{6}(K^{2}+L^{2})}{K}\left(\mathfrak{a}_{1}U
+A_{4}+\mathfrak{a}_{1}A_{2}\right)\nonumber\\
&&+ \ \frac{C^{\frac{p+3}{2}}p^{2p}(K^{2}+L^{2})}{K}\mathfrak{a}_{2}\int|\nabla\phi|^{2}
\phi^{2p-2}dV_{t}+\frac{C^{\frac{p}{2}}p^{p}(K^{2}+L^{2})}{K}\mathfrak{a}_{3}
\int\phi^{2p}dV_{t}.\label{3.19}
\end{eqnarray}

In the following we will prove a local curvature estimate so we may without loss of generality assume that $M$ is a complete
manifold ($M$ may not be compact). Assume now that
\begin{equation}
|{\rm Ric}_{g(t)}|_{g(t)}\leq K, \ \ \ |\alpha(t)|_{g(t)}
\leq L \ \ \ \text{on} \ \Omega=B_{g(0)}(x_{0},\rho/\sqrt{K})\label{3.20}
\end{equation}
on $\Omega\times[0,T]$, where $\rho, K, L$ are positive constants and
$x_{0}\in M$, and $\Omega$ is compactly contained in $M$. Consider the cutoff function
\begin{equation}
\phi:=\left(\frac{\rho/\sqrt{K}-d_{g(0)}(x_{0},\cdot)}{\theta\rho/\sqrt{K}}
\right)_{+},\label{3.21}
\end{equation}
where $\theta\geq1$ is any positive constant. Then
$$
e^{-2(K+|a|+|b|L)t}g(0)
\leq g(t)\leq e^{2(K+|a|+|b|L)t}g(0)
$$
and
$$
|\nabla_{g(t)}\phi|_{g(t)}
\leq e^{(K+|a|+|b|L)T}|\nabla_{g(0)}\phi|_{g(0)}
\leq\frac{\sqrt{K}}{\theta\rho}e^{(K+|a|+|b|L)T}
$$
for any $t\in[0,T]$. Set
\begin{equation}
K':=K+|a|+|b|L.\label{3.22}
\end{equation}
Then
\begin{equation}
e^{-2K't}g(0)\leq g(t)\leq e^{2K't}g(0), \ \ \
|\nabla_{g(t)}\phi|_{g(t)}\leq\frac{\sqrt{K'}}{\theta\rho}e^{K'T}.
\label{3.23}
\end{equation}
By Young's inequality we have
\begin{eqnarray*}
A_{4}&=&\int_{\mathcal{M}}|{\rm Rm}|^{p-1}
|\nabla\phi|^{2}\phi^{2p-2}dV_{t}\\
&\leq&\int_{B_{g(0)}(x_{0},\rho/\sqrt{K})}
|{\rm Rm}|^{p-1}\phi^{2p-2}\frac{K'}{\theta^{2}\rho^{2}}e^{2K'T}dV_{t}\\
&\leq&\int_{B_{g(0)}(x_{0},\rho/\sqrt{K})}
\left[\frac{(|{\rm Rm}|^{p-1}\phi^{2p-2})^{\frac{p}{p-1}}}{\frac{p}{p-1}}
+\frac{(K'(\theta\rho)^{-2}e^{2K'T})^{p}}{p}\right]dV_{t}\\
&\leq&A_{1}+\frac{K'{}^{p}e^{2K'pT}}{p}(\theta\rho)^{-2p}
{\rm Vol}_{g(t)}\left(B_{g(0)}\left(x_{0},\frac{\rho}{\sqrt{K}}\right)\right)\\
&\leq&U+K'{}^{p}(\theta\rho)^{-2p}e^{2K'pT}{\rm Vol}_{g(t)}
\left(B_{g(0)}\left(x_{0},\frac{\rho}{\sqrt{K}}\right)\right).
\end{eqnarray*}
The inequality (\ref{3.23}) yields
$$
\int_{\mathcal{M}}|\nabla\phi|^{2}\phi^{2p-2}dV_{t}\leq K'\theta^{-2p}\rho^{-2}e^{2K'T}{\rm Vol}_{g(t)}
\left(B_{g(0)}\left(x_{0},\frac{\rho}{\sqrt{K}}\right)\right)
$$
and
$$
\int_{\mathcal{M}}\phi^{2p}dV_{t}\leq \theta^{-2p}{\rm Vol}_{g(t)}
\left(B_{g(0)}\left(x_{0},\frac{\rho}{\sqrt{K}}\right)\right).
$$
Therefore from (\ref{3.19}) and $A_{2}\leq \theta^{-1}A_{4}
\leq A_{4}$, we obtain
\begin{eqnarray*}
U'&\leq&\frac{Cp^{6}(K^{2}+L^{2})}{K}\mathfrak{a}_{1}U+\frac{C^{\frac{p+3}{2}}p^{2p}(K^{2}+L^{2})}{\theta^{2p}K}\\
&&\times\left(\mathfrak{a}_{1}K'{}^{p}\rho^{-2p}e^{2K'pT}
+\mathfrak{a}_{2}K'\rho^{-2}e^{2K'T}+\mathfrak{a}_{3}\right)
{\rm Vol}_{g(t)}\left(B_{g(0)}
\left(x_{0},\frac{\rho}{\sqrt{K}}\right)\right).
\end{eqnarray*}
According to Young's inequality, we have
\begin{eqnarray*}
\mathfrak{a}_{2}K'\rho^{-2}e^{2K'T}&=&
\left(p^{1/p}K'\rho^{-2}e^{2K'T}\right)\left(\frac{\mathfrak{a}_{2}}{p^{1/p}}
\right)\\
&\leq&K'{}^{p}\rho^{-2p}e^{2K'pT}
+\frac{(\mathfrak{a}_{2}/p^{1/p})^{q}}{q}\\
&\leq&K'{}^{p}\rho^{-2p}e^{2K'pT}+\mathfrak{a}^{q}_{2}\\
&\leq&\mathfrak{a}_{1}K'{}^{p}\rho^{-2p}e^{2K'pT}+(\mathfrak{a}_{2}\vee 1)^{2},
\end{eqnarray*}
where $q=\frac{p}{p-1}\in(1,2)$ (because $p\geq3$). Consequently
\begin{eqnarray}
\frac{d}{dt}U(t)&\leq&\frac{C^{\frac{p+3}{2}}p^{2p}(K^{2}+L^{2})}{
\theta^{2p}K}\left[\mathfrak{a}_{1}K'{}^{p}\rho^{-2p}e^{2K'pT}
+(\mathfrak{a}_{2}\vee1)^{2}+\mathfrak{a}_{3}\right]\nonumber\\
&&\times{\rm Vol}_{g(t)}
\left(B_{g(0)}\left(x_{0},\frac{\rho}{\sqrt{K}}\right)\right)+\frac{Cp^{6}(K^{2}+L^{2})}{K}
\mathfrak{a}_{1}
U(t).\label{3.24}
\end{eqnarray}
For convenience, we also introduce
\begin{eqnarray*}
A&:=&C p^{6}\left(1+\frac{K^{2}+L^{2}}{K}\right)\mathfrak{a}_{1},\\
B&:=&\frac{C^{\frac{p+3}{2}}p^{2p}(K^{2}+L^{2})}{\theta^{2p}K}
\left[\mathfrak{a}_{1}K'{}^{p}\rho^{-2p}e^{2K'pT}
+(\mathfrak{a}_{2}\vee1)^{2}+\mathfrak{a}_{3}\right].
\end{eqnarray*}
Now the inequality (\ref{3.24}) becomes
$$
U'(t)\leq AU(t)+B{\rm Vol}_{g(t)}\left(B_{g(0)}
\left(x_{0},\frac{\rho}{\sqrt{K}}\right)\right)
$$
and then
\begin{equation}
e^{-At}U(t)\leq U(0)+\int^{t}_{0}Be^{-A\tau}{\rm Vol}_{g(\tau)}
\left(B_{g(0)}\left(x_{0},\frac{\rho}{\sqrt{K}}\right)\right)d\tau.\label{3.25}
\end{equation}
Using (\ref{3.23}) we have that for any $\tau\in[0,t]$,
$$
g(\tau)\leq e^{2K'\tau}g(0)\leq e^{2K'\tau}e^{2K't}g(t)
\leq e^{4K'T}g(t)
$$
and hence
$$
{\rm Vol}_{g(\tau)}\left(B_{g(0)}\left(x_{0},\frac{\rho}{\sqrt{K}}\right)\right)
\leq e^{2nK'T}{\rm Vol}_{g(t)}\left(B_{g(0)}\left(x_{0},\frac{\rho}{\sqrt{K}}
\right)\right).
$$
Plugging into (\ref{3.25}), it follows that
\begin{eqnarray}
U(t)&\leq&e^{At}\left[ U(0)+B e^{2nK'T}{\rm Vol}_{g(t)}
\left(B_{g(0)}\left(x_{0},\frac{\rho}{\sqrt{K}}\right)\right)\int^{t}_{0}
e^{-A\tau}d\tau\right]\nonumber\\
&\leq&e^{AT}\left[U(0)+\frac{B}{A}e^{2nK'T}\left(1-e^{-At}\right){\rm Vol}_{g(t)}
\left(B_{g(0)}\left(x_{0},\frac{\rho}{\sqrt{K}}\right)\right)\right]\label{3.26}\\
&\leq&e^{AT}\left[U(0)+\frac{B}{A}e^{2nK'T}{\rm Vol}_{g(t)}
\left(B_{g(0)}\left(x_{0},\frac{\rho}{\sqrt{K}}\right)\right)\right].\nonumber
\end{eqnarray}
Moreover we get
$$
U(t)\leq e^{AT}\left[U(0)
+\frac{B}{A}e^{4nK'T}{\rm Vol}_{g(T)}\left(
B_{g(0)}\left(x_{0},\frac{\rho}{\sqrt{K}}\right)\right)\right], \ \ \ t\in[0,T].
$$
The last step is to estimate the initial data $U(0)$:
\begin{eqnarray*}
U(0)&\leq &\int|{\rm Rm}_{g(0)}|^{p}_{g(0)}
\phi^{2p}dV_{g(0)}\\
&&+ \ \frac{1}{2K}\int|{\rm Ric}_{g(0)}|^{2}_{g(0)}
|{\rm Rm}_{g(0)}|^{p-1}_{g(0)}\phi^{2p}dV_{g(0)}\\
&&+ \ \frac{C p^{6}(K^{2}+L^{2})}{K(p-1)}\int|{\rm Rm}_{g(0)}|^{p-1}_{g(0)}
\phi^{2p}dV_{g(0)}\\
&&+ \ \frac{C p^{2}(K^{2}+L^{2})}{KL^{2}}\int|\alpha(0)|^{2}_{g(0)}|{\rm Rm}_{g(0)}|^{p-1}_{g(0)}
\phi^{2p}dV_{g(0)}\\
&&+ \ \frac{C^{\frac{p+3}{2}}p^{2p}(K^{2}+L^{2})}{KL^{2}}
\int|\alpha(0)|^{2}\phi^{2p}dV_{g(0)}\\
&\leq&\int|{\rm Rm}_{g(0)}|^{p}_{g(0)}
\phi^{2p}dV_{g(0)}\\
&&+ \ \frac{C p^{6}(K^{2}+L^{2})}{K}\int|{\rm Rm}_{g(0)}|^{p-1}_{g(0)}
\phi^{2p}dV_{g(0)}\\
&&+ \ \frac{C^{\frac{p+3}{2}}p^{2p}(K^{2}+L^{2})}{K}
\int\phi^{2p}dV_{g(0)}\\
&\leq&\int|{\rm Rm}_{g(0)}|^{p}_{g(0)}\phi^{2p}dV_{g(0)}\\
&&+ \ \frac{C^{\frac{p+3}{2}}p^{2p}(K^{2}
+L^{2})}{K}\int\phi^{2p}dV_{g(0)}\\
&&+ \ \frac{C p^{6}(K^{2}+L^{2})}{K}\left[\frac{p-1}{p}\int|{\rm Rm}_{g(0)}|^{p}_{g(0)}
\phi^{2p}dV_{g(0)}+\frac{1}{p}\int\phi^{2p}dV_{g(0)}\right]\\
&\leq& Cp^{6}\left(1+\frac{K^{2}+L^{2}}{K}\right)\int|{\rm Rm}_{g(0)}|^{p}_{g(0)}
\phi^{2p}dV_{g(0)}\\
&&+ \ \frac{C^{\frac{p+3}{2}}p^{2p}(K^{2}+L^{2})}{K
\theta^{2p}}{\rm Vol}_{g(0)}
\left(B_{g(0)}\left(x_{0},\frac{\rho}{\sqrt{K}}\right)\right)\\
&\leq& Cp^{6}\left(1+\frac{K^{2}+L^{2}}{K}\right)\int|{\rm Rm}_{g(0)}|^{p}_{g(0)}
\phi^{2p}dV_{g(0)}\\
&&+ \ \frac{C^{\frac{p+3}{2}}p^{2p}(K^{2}+L^{2})}{K
\theta^{2p}}e^{nK'T}{\rm Vol}_{g(t)}
\left(B_{g(0)}\left(x_{0},\frac{\rho}{\sqrt{K}}\right)\right).
\end{eqnarray*}
Together with (\ref{3.26}), we arrive at
\begin{eqnarray}
U(t)&\leq&e^{AT}\bigg[A\int|{\rm Rm}_{g(0)}|^{p}_{g(0)}\phi^{2p}dV_{g(0)}
+{\rm Vol}_{g(t)}\left(B_{g(0)}\left(x_{0},\frac{\rho}{\sqrt{K}}\right)\right)\nonumber\\
&&\times \ \left(\frac{C^{\frac{p+3}{2}}p^{2p}(K^{2}+L^{2})}{K\theta^{2p}}e^{nK'T}+\frac{B}{A}e^{2nK'T}\right)
\bigg]\nonumber\\
&\leq& e^{AT}\bigg[Cp^{6}\left(1+\frac{K^{2}+L^{2}}{K}\right)
\mathfrak{a}_{1}\int|{\rm Rm}_{g(0)}|^{p}_{g(0)}
\phi^{2p}dV_{g(0)}\label{3.27}\\
&&+ \ \left(\frac{B}{A}+\frac{C^{\frac{p+3}{2}}p^{2p}}{\theta^{2p}}\frac{K^{2}+L^{2}}{K}\right)e^{2nK'T}
{\rm Vol}_{g(t)}\left(B_{g(0)}\left(x_{0},\frac{\rho}{\sqrt{K}}\right)\right)\bigg].\nonumber
\end{eqnarray}
In particular, for any $\tau>1$, one has
\begin{eqnarray}
&&\int_{B_{g(0)}(x_{0},\rho/\tau\sqrt{K})}
|{\rm Rm}_{g(t)}|^{p}_{g(t)}dV_{g(t)}\label{3.28}\\
&\leq&\left(\frac{\tau}{\tau-1}\theta\right)^{2p}e^{AT}\bigg[\frac{Cp^{6}}{\theta^{2p}}\left(1+\frac{K^{2}+L^{2}}{K}\right)\mathfrak{a}_{1}
\int_{B_{g(0)}(x_{0},\rho/\sqrt{K})}
|{\rm Rm}_{g(0)}|^{p}_{g(0)}dV_{g(0)}\nonumber\\
&&+ \ \left(\frac{B}{A}+\frac{C^{\frac{p+3}{2}}p^{2p}}{\theta^{2p}}\frac{K^{2}+L^{2}}{K}\right)
e^{2nK'T}{\rm Vol}_{g(t)}
\left(B_{g(0)}\left(x_{0},\frac{\rho}{\sqrt{K}}\right)\right)\bigg].\nonumber
\end{eqnarray}
Fixing the volume of the ball $B_{g(0)}(x_{0},\rho/\sqrt{K})$, we have
\begin{eqnarray}
&&\int_{B_{g(0)}(x_{0},\rho/\tau\sqrt{K})}
|{\rm Rm}_{g(t)}|^{p}_{g(t)}dV_{g(t)}\label{3.29}\\
&\leq&\left(\frac{\tau}{\tau-1}\theta\right)^{2p}e^{AT}\bigg[
\frac{Cp^{6}}{\theta^{2p}}\left(1+\frac{K^{2}+L^{2}}{K}\right)\mathfrak{a}_{1}
\int_{B_{g(0)}(x_{0},\rho/\sqrt{K})}
|{\rm Rm}_{g(0)}|^{p}_{g(0)}dV_{g(0)}\nonumber\\
&&+ \ \left(\frac{B}{A}+\frac{C^{\frac{p+3}{2}}p^{2p}}{\theta^{2p}}\frac{K^{2}+L^{2}}{K}\right)
e^{4nK'T}{\rm Vol}_{g(T)}
\left(B_{g(0)}\left(x_{0},\frac{\rho}{\sqrt{K}}\right)\right)\bigg].\nonumber
\end{eqnarray}

%%%%%%%%%%%%%%%%%%%%%%%%%%%%%%%%%%%%%%%%%%%%%%%%%%%%%%%%%%%
\subsection{Proof of Theorem \ref{t2.1}}\label{subsection3.3}
%%%%%%%%%%%%%%%%%%%%%%%%%%%%%%%%%%%%%%%%%%%%%%%%%%%%%%%%%%%

Choosing $\tau=2$ and $\theta=1$ yields
\begin{eqnarray*}
&&\int_{B_{g(0)}(x_{0},\rho/2\sqrt{K})}
|{\rm Rm}_{g(t)}|^{p}_{g(t)}dV_{g(t)}\\
&\leq&2^{2p}e^{AT}\bigg[
Cp^{6}\left(1+\frac{K^{2}+L^{2}}{K}\right)\mathfrak{a}_{1}
\int_{B_{g(0)}(x_{0},\rho/\sqrt{K})}
|{\rm Rm}_{g(0)}|^{p}_{g(0)}dV_{g(0)}\nonumber\\
&&+ \ \left(\frac{B}{A}+C^{\frac{p+3}{2}}p^{2p}\frac{K^{2}+L^{2}}{K}\right)
e^{4nK'T}{\rm Vol}_{g(T)}
\left(B_{g(0)}\left(x_{0},\frac{\rho}{\sqrt{K}}\right)\right)\bigg].\nonumber
\end{eqnarray*}
According to the volumes relations
$$
{\rm Vol}_{g(T)}\left(B_{g(0)}
\left(x_{0},\frac{\rho}{\sqrt{K}}\right)\right)
\leq e^{nK'T}{\rm Vol}_{g(0)}
\left(B_{g(0)}\left(x_{0}, \frac{\rho}{\sqrt{K}}\right)\right)
$$
we obtain
\begin{eqnarray*}
&&\int_{B_{g(0)}(x_{0},\rho/2\sqrt{K})}
|{\rm Rm}_{g(t)}|^{p}_{g(t)}dV_{g(t)}\\
&\leq&2^{2p}e^{AT}\bigg[
Cp^{6}\left(1+\frac{K^{2}+L^{2}}{K}\right)\mathfrak{a}_{1}
\int_{B_{g(0)}(x_{0},\rho/\sqrt{K})}
|{\rm Rm}_{g(0)}|^{p}_{g(0)}dV_{g(0)}\nonumber\\
&&+ \ \left(\frac{B}{A}+C^{\frac{p+3}{2}}p^{2p}\frac{K^{2}+L^{2}}{K}\right)
e^{5nK'T}{\rm Vol}_{g(T)}
\left(B_{g(0)}\left(x_{0},\frac{\rho}{\sqrt{K}}\right)\right)\bigg],\nonumber
\end{eqnarray*}
where
$$
\mathfrak{a}_{1}=1+K+L+\frac{K}{p^{4}L^{2}}, \ \ \ \mathfrak{a}_{2}=
\frac{L^{p-1}}{K^{\frac{p-3}{2}}}, \ \ \ \mathfrak{a}_{3}=
\frac{(1+K+L)(L\vee 1)^{p+2+\frac{1}{p-2}}}{(K\wedge 1)^{\frac{1}{2}(
p-2+\frac{1}{p-2})}}
$$
and
$$
A=Cp^{6}\left(1+\frac{K^{2}+L^{2}}{K}\right)\mathfrak{a}_{1}
$$
and
$$
B=C^{\frac{p+3}{2}}p^{2p}\frac{K^{2}+L^{2}}{K}
\left[\mathfrak{a}_{1}K'{}{p}\rho^{-2p}
e^{2K'pT}+(\mathfrak{a}_{2}\vee 1)^{2}+\mathfrak{a}_{3}\right].
$$
Consequently
\begin{eqnarray*}
\int_{B_{g(0)}(x_{0},\rho/2\sqrt{K})}
|{\rm Rm}_{g(t)}|^{p}_{g(t)}dV_{g(t)}&\leq&C e^{CT}
\int_{B_{g(0)}(x_{0},\rho/\sqrt{K})}|{\rm Rm}_{g(0)}|^{p}_{g(0)}dV_{g(0)}+\nonumber\\
&&C(1+\rho^{-2p})e^{CT}{\rm Vol}_{g(0)}
\left(B_{g(0)}\left(x_{0},\frac{\rho}{\sqrt{K}}\right)\right)
\end{eqnarray*}
for some constant $C=C(K, L,n, p)$. As in \cite{KMW}, by the Bishop-Gromov volume comparison theorem, we have
$$
\left(\frac{1}{{\rm Vol}_{g(0)}(B_{g(0)}(x_{0},\rho/2\sqrt{K}))}
\int_{B_{g(0)}(x_{0},\rho/2\sqrt{K})}
|{\rm Rm}_{g(t)}|^{p}_{g(t)}dV_{g(t)} \right)^{1/p}
$$
\begin{equation}
\leq \ \ Ce^{C(T+\rho)}\left[\Lambda_{0}+(1+\rho^{-2})\right],\label{3.30}
\end{equation}
with
$$
\Lambda_{0}=\sup_{B_{g(0)}(x_{0},\rho/\sqrt{K})}|{\rm Rm}_{g(0)}|.
$$

Recall the following evolution inequalities proved almost in the same way in \cite{LYZ}:
\begin{eqnarray*}
(\partial_{t}-\Delta)|{\rm Rm}|^{2}&\leq&-2|\nabla{\rm Rm}|^{2}
+C|{\rm Rm}|^{2}+C|{\rm Rm}|^{3}+C|{\rm Rm}|^{2}|\alpha|
+C|{\rm Rm}||\nabla^{2}\alpha|\\
&\leq&-2|\nabla{\rm Rm}|^{2}+C(1+L)|{\rm Rm}|^{2}+C|{\rm Rm}|^{3}
+|\nabla^{2}\alpha|^{2},\\
(\partial_{t}-\Delta)|\alpha|^{2}&\leq&-2|\nabla^{2}\alpha|^{2}
+C|\nabla\alpha|^{2}+C|{\rm Rm}||\nabla\alpha|^{2}\\
&&+ \ C|\alpha||\nabla\alpha||\nabla{\rm Rm}|+C|\alpha||\nabla\alpha|^{2}\\
&\leq&-2|\nabla^{2}\alpha|^{2}+|\nabla{\rm Rm}|^{2}+C|{\rm Rm}||\nabla\alpha|^{2}+C(1+L+L^{2})|\nabla\alpha|^{2}.
\end{eqnarray*}
Hence
\begin{eqnarray*}
(\partial_{t}-\Delta)\left(|{\rm Rm}|^{2}+|\alpha|^{2}\right)
&\leq&C(1+L+L^{2}+|{\rm Rm}|)\left(|{\rm Rm}|^{2}+|\nabla\alpha|^{2}\right).
\end{eqnarray*}
Write
$$
u:=|{\rm Rm}|^{2}+|\alpha|^{2}, \ \ \ f:=1+L+L^{2}+|{\rm Rm}|.
$$
Then we get
$$
(\partial_{t}-\Delta)u\leq Cfu
$$
which is the same as (3.3) of \cite{KMW}. Following exactly the same 
argument on pages 2620-2623 of \cite{KMW}, together with (\ref{3.30}), we 
prove Theorem \ref{t2.1}.

\begin{remark}\label{r3.6} In \cite{L2020}, the first author
applies the method of proving Theorem \ref{t2.1} to the
Ricci-harmonic flow and weakens the condition of Theorem 2.7 in \cite{L}.
\end{remark}

%%%%%%%%%%%%%%%%%%%%%%%%%%%%%%%%%%%%%%%%%%%%%%%%%%%%%%%%%%%%%%%%%%%%%%%%%%%%%%
%%%%%%%%%%%%%%%%%%%%%%%%%%%%%%%%%%%%%%%%%%%%%%%%%%%%%%%%%%%%%%%%%%%%%%%%%%%%%%
%%%%%%%%%%%%%%%%%%%%%%%%%%%%%%%%%%%%%%%%%%%%%%%%%%%%%%%%%%%%%%%%%%%%%%%%%%%%%%

%%%%%%%%%%%%%%%%%%%%%%%%%%%%%%%%%%%%%%%%%%%%%%%%%%%%%%%%%%%%%%%%%%%%%%%%%%%%%%
%%%%%%%%%%%%%%%%%%%%%%%%%%%%%%%%%%%%%%%%%%%%%%%%%%%%%%%%%%%%%%%%%%%%%%%%%%%%%%
%%%%%%%%%%%%%%%%%%%%%%%%%%%%%%%%%%%%%%%%%%%%%%%%%%%%%%%%%%%%%%%%%%%%%%%%%%%%%%

%%%%%%%%%%%%%%%%%%%%%%%%%%%%%%%%%%%%%%%%%%%%%%%%%%%%%%%%%%%%%%%%%%%%%%%%%%%%%%
\end{document}